\documentclass[a4paper,reqno]{amsart}

\usepackage[T1]{fontenc}
\usepackage[utf8x]{inputenc}
\usepackage[english]{babel}
\usepackage{yfonts}
\usepackage{dsfont}
\usepackage{amscd,amssymb,amsmath,amsthm,amsfonts}
\usepackage{mathtools}
\usepackage{accents}
\usepackage{tensor}
\usepackage{graphicx}
\usepackage{tikz}
\usetikzlibrary{calc,matrix,arrows,decorations.pathmorphing}
\usepackage{calc}
\usepackage[cal=boondox,scr=boondoxo]{mathalfa}
\usepackage{enumitem}
\usepackage{marginnote}
\usepackage{booktabs}
\usepackage{scalerel}[2016/12/29]

\usepackage{hyperref}
\hypersetup{colorlinks=true,pageanchor=false,
linkcolor=blue,citecolor=red,urlcolor=red}
\usepackage[all]{hypcap}

\allowdisplaybreaks

\newtheorem{theorem}{Theorem}[section]

\newtheorem{proposition}[theorem]{Proposition}
\newtheorem{lemma}[theorem]{Lemma}

\theoremstyle{definition}
\newtheorem{definition}[theorem]{Definition}

\theoremstyle{remark}
\newtheorem{remark}[theorem]{Remark}

\newcommand{\N}{\mathbb{N}}
\newcommand{\Z}{\mathbb{Z}}

\newcommand{\C}{\mathbb{C}}

\newcommand{\fraku}{\mathfrak{u}}

\newcommand{\id}{\mathrm{id}}

\newcommand{\rmspan}{\operatorname{span}}

\newcommand{\leqs}{\leqslant}
\newcommand{\geqs}{\geqslant}

\newcommand{\mods}[1]{\operatorname{\mathnormal{#1}-mod}}

\newcommand{\fsl}{\mathfrak{sl}}

\newcommand{\Hom}{\mathrm{Hom}}

\newcommand{\End}{\mathrm{End}}

\newcommand{\TL}{\mathrm{TL}}
\newcommand{\tTL}{\tilde{\mathrm{TL}}}
\newcommand{\bTL}{\bar{\mathrm{TL}}}

\newcommand{\up}{+}
\newcommand{\down}{-}

\makeatletter
\newcommand{\subalign}[1]{
  \vcenter{
    \Let@ \restore@math@cr \default@tag
    \baselineskip\fontdimen10 \scriptfont\tw@
    \advance\baselineskip\fontdimen12 \scriptfont\tw@
    \lineskip\thr@@\fontdimen8 \scriptfont\thr@@
    \lineskiplimit\lineskip
    \ialign{\hfil$\m@th\scriptstyle##$&$\m@th\scriptstyle{}##$\crcr
      #1\crcr
    }
  }
}
\makeatother

\def\clap#1{\hbox to 0pt{\hss#1\hss}}

\def\mathrlap{\mathpalette\mathrlapinternal}

\def\mathrlapinternal#1#2{%
\rlap{$\mathsurround=0pt#1{#2}$}}

\newcommand{\tl}[1]{\raisebox{-0.5\height + 2.5pt}{\includegraphics{tl_#1.pdf}}}
\newcommand{\domination}[1]{\raisebox{-0.5\height + 2.5pt}{\includegraphics{domination_#1.pdf}}}
\newcommand{\fusion}[1]{\raisebox{-0.5\height + 2.5pt}{\includegraphics{fusion_#1.pdf}}}
\newcommand{\etl}[1]{\raisebox{-0.5\height + 2.5pt}{\includegraphics{etl_#1.pdf}}}

\newcommand\arxiv[2]{\href{http://arXiv.org/abs/#1}{\texttt{arXiv:\allowbreak #1} #2}}
\newcommand\doi[2]{\href{http://doi.org/#1}{#2}}
  
\begin{document}

\raggedbottom

\title{Diagrammatic Construction of Representations of Small Quantum $\fsl_2$}

\author[C. Blanchet]{Christian Blanchet} 
\address{Institut de Mathématiques de Jussieu -- Paris Rive Gauche, UMR 7586, Université Paris Diderot -- Paris 7, Bâtiment Sophie Germain, 75205 Paris Cedex 13, France} 
\email{christian.blanchet@imj-prg.fr}

\author[M. De Renzi]{Marco De Renzi} 
\address{Department of Mathematics, Faculty of Science and Engineering, Waseda University, 3-4-1 \={O}kubo, Shinjuku-ku, Tokyo, 169-8555, Japan} 
\email{m.derenzi@kurenai.waseda.jp}
\address{Institute of Mathematics, University of Zurich, Winterthurerstrasse 190, CH-8057 Zurich, Switzerland}
\email{marco.derenzi@math.uzh.ch}

\author[J. Murakami]{Jun Murakami} 
\address{Department of Mathematics, Faculty of Science and Engineering, Waseda University, 3-4-1 \={O}kubo, Shinjuku-ku, Tokyo, 169-8555, Japan} 
\email{murakami@waseda.jp}

\begin{abstract}
 We provide a combinatorial description of the monoidal category generated by the fundamental representation of the small quantum group of $\fsl_2$ at a root of unity $q$ of odd order. Our approach is diagrammatic, and it relies on an extension of the Temperley--Lieb category specialized at $\delta = -q-q^{-1}$.
\end{abstract}

\maketitle
\setcounter{tocdepth}{3}

\section{Introduction}

The main goal of this paper is, very roughly speaking, to find a diagrammatic description of the category $\mods{\bar{U}}$ of finite-dimensional left $\bar{U}$-modules, where ${\bar{U} = \bar{U}_q \fsl_2}$ is a finite-dimensional Hopf algebra called the \textit{small quantum group of $\fsl_2$}. First introduced by Lusztig in \cite{L90}, $\bar{U}$ arises from a quantum deformation of the universal enveloping algebra of $\fsl_2$ with deformation parameter $q = e^{\smash{\frac{2 \pi i}{r}}}$, where $3 \leqs r \in \Z$ is an odd integer called the \textit{level}. The motivation behind this work is topological: indeed, $\bar{U}$ is a factorizable ribbon Hopf algebra, which means it can be used as an algebraic tool for building non-semisimple TQFTs \cite{DGP17}. At present, constructions rely either on the structure of Hopf algebras, or on some general categorical machinery \cite{DGGPR19}. Although this theory produces intreresting new phenomena, which differ strikingly from the semisimple case \cite{DGGPR20}, it exploits a rather elaborate technical setup, which can require some time to fully digest. However, in the case of the small quantum group $\bar{U}$, we can look for a more combinatorial approach by figuring out an explicit description of a large category of $\bar{U}$-modules in diagrammatic terms. This idea is pursued in \cite{DM20}, where a skein theoretic reformulation of the topological invariants of closed 3-manifold constructed from $\bar{U}$ is obtained, based on results provided here.

The starting point for our work is a diagrammatic description of the category of representations of a different Hopf algebra $U = U_q \fsl_2$ called the \textit{full quantum group}, or the \textit{divided power quantum group} of $\fsl_2$. This version of quantum $\fsl_2$ is infinite-dimensional, and its definition is also due to Lusztig. In fact, the small quantum group $\bar{U}$ was originally introduced as a finite-dimensional Hopf subalgebra of $U$. It was first shown in \cite{MM92} that $\End_{U}(X^{\otimes m})$ is isomorphic to the \textit{$m$th Temperley--Lieb algebra} $\TL(m)$ of \cite{TL71} with parameter $\delta = -A^2 - A^{-2}$ specialized at $A = q^{\frac{r+1}{2}}$, where $X$ is the fundamental simple $U$-module of dimension 2, see also \cite{FK97,AST15}. Since the small quantum group $\bar{U}$ is a Hopf subalgebra of $U$, we have a natural restriction functor from $\mods{U}$ to $\mods{\bar{U}}$. However, there exist indecomposable $U$-modules which become decomposable when considered as $\bar{U}$-modules, and $\End_{\bar{U}}(X^{\otimes m})$ is not isomorphic to $\TL(m)$ in general. 

The question of understanding $\End_{\bar{U}}(X^{\otimes m})$ has been studied in greater detail in the case of even roots of unity, when the deformation parameter is $q = e^{\frac{2 \pi i}{r}}$ for some even integer $4 \leqs r \in \Z$. Indeed, in this case $U$ contains a finite-dimensional Hopf subalgebra $\bar{U}$ which is called the \textit{restricted quantum group of $\fsl_2$}. In \cite{GST12}, it is shown that an extension of the $m$th Temperley--Lieb algebra $\TL(m)$, called the \textit{lattice W-algebra}, is isomorphic to $\End_{\bar{U}}(X^{\otimes m})$ for every $m \geqs 0$ when $r = 4$, and it is conjectured that the isomorphism holds in general for every even $r \geqs 4$. In \cite{M17}, generators and relations for a description of $\End_{\bar{U}}(X^{\otimes m})$ in terms of planar diagrams are provided for every $m \geqs 0$ and every even $r \geqs 4$, and it is conjectured these give a complete presentation. However, for even values of $r$, the Hopf algebra $\bar{U}$ is not ribbon, and although it admits a ribbon extension, the latter is not factorizable. In particular, the machinery of \cite{DGP17} only produces quantum invariants of closed $3$-manifolds, not TQFTs. This is why we turn our attention to the less discussed odd level case. In order to do this, we consider an extension $\tTL$ of the Temperley--Lieb category $\TL$, which is obtained by introducing four additional generating morphisms, as well as several relations between them. These additional generators correspond to the non-trivial splitting, in $\mods{\bar{U}}$, of a certain indecomposable $U$-module. The definition of $\tTL$, which is given in Section~\ref{S:extended_TL_category}, allows us to establish our main result.

\begin{theorem}\label{T:functor}
 There exists a full monoidal linear functor
 \[
  F_{\tTL} : \tTL \to \mods{\bar{U}}.
 \]
\end{theorem}

This allows us to describe in diagrammatic terms every morphism in the full monoidal subcategory of $\mods{\bar{U}}$ generated by $X$.

\subsection*{Structure of the paper}

In Section~\ref{S:small_quantum_group}, we recall the definition of the small quantum group $\bar{U}$ and the main properties of its category of finite-dimensional representations $\mods{\bar{U}}$. In particular, we collect definitions of simple and indecomposable projective $\bar{U}$-modules, and we give explicit formulas for $\bar{U}$-module morphisms between them. In Section~\ref{S:TL_category}, we recall the definition of the Temperley--Lieb category $\TL_A$ of indeterminate $A$, and within it we highlight a special family of idempotent endomorphisms $g_m \in \TL_A(m) = \TL_A(m,m)$ for $r \leqs m \leqs 2r-2$. Under the specialization $A = q^{\frac{r+1}{2}}$, these planar diagrams, which were first considered in \cite{GW93}, extend the standard family of (\textit{simple}) \textit{Jones--Wenzl idempotents} $f_m \in \TL_A(m)$ of \cite{J83,W87}, which are well-defined for $0 \leqs m \leqs r-1$. Since the endomorphisms $g_m$ can be interpreted in terms of indecomposable projective $\bar{U}$-modules, we refer to them as \textit{non-semisimple Jones--Wenzl idempotents}. Our explicit recursive formulas provide an odd level counterpart to the work of Ibanez \cite{I15} and Moore \cite{M18}, who independently came up with similar ones for the restricted quantum group of $\fsl_2$ at even roots of unity. 
In Section~\ref{S:monoidal_linear_functor}, we recall the definition of a well-known monoidal linear functor $F_{\TL} : \TL \to \mods{\bar{U}}$, where $\TL$ denotes the category obtained from $\TL_A$ under the specialization $A = q^{\frac{r+1}{2}}$, and we explicitly compute the image of simple and non-semisimple Jones--Wenzl idempotents. In doing so, we show that the functor $F_{\TL}$ is not full. In Section~\ref{S:extended_TL_category}, we define the extended Temperley--Lieb category $\tTL$ by introducing four additional generating morphisms to $\TL$, and we extend $F_{\TL}$ to a functor $F_{\tTL} : \tTL \to \mods{\bar{U}}$. This allows us to prove our main result in Section~\ref{S:proof}. We do so in two steps: First, we show that the union of simple and non-semisimple Jones--Wenzl idempotents dominates the category $\tTL$; Then, we exhibit explicit morphisms which ensure fullness of $F_{\tTL}$. In Section~\ref{S:additional_relations}, we list additional relations satisfied by generating morphisms of $\tTL$ in the quotient $\bTL = \tTL / \ker F_{\tTL}$, which will be useful in future works. It remains an open question to understand whether these relations provide a complete presentation of $\bTL$, in the sense that they generate the kernel of $F_{\tTL}$. However, it seems likely that further relations, generalizing equation~\eqref{E:p_i} to the case $r \leqs m \leqs 2r-2$, will be needed.

\subsection*{Acknowledgments}

This work was supported by JSPS Fellowship for Research in Japan ID~PE18705.

\section{Small quantum group}\label{S:small_quantum_group}

In this section we recall the definition, due to Lusztig, of the small quantum group of $\fsl_2$ at odd roots of unity. To start, let us fix an odd integer $3 \leqs r \in \Z$ and let us consider the primitive $r$th root of unity $q = e^{\frac{2 \pi i}{r}}$. For every integer $k \geqs 0$ we introduce the notation 
\[
 \{ k \} := q^k - q^{-k},
 \quad \{ k \}' := q^k + q^{-k},
 \quad [k] := \frac{\{ k \}}{\{ 1 \}},
 \quad [k]! := \prod_{j=1}^k [j].
\]
We denote with $\bar{U}$ the algebra over $\C$ with generators $\{ E,F,K \}$ and relations
\begin{gather*}
 E^r = F^r = 0, \quad K^r = 1, \\*
 K E K^{-1} = q^2 E, \quad K F K^{-1} = q^{-2} F, \quad [E,F] = \frac{K - K^{-1}}{q-q^{-1}}.
\end{gather*}
This algebra was first introduced in \cite[Section~6.5]{L90} under the notation $\tilde{\fraku}$. We can make $\bar{U}$ into a Hopf algebra, called the \textit{small quantum group of $\fsl_2$}, by setting
\begin{align*}
 \Delta(E) &= E \otimes K + 1 \otimes E, & \varepsilon(E) &= 0, & S(E) &= -E K^{-1}, \\*
 \Delta(F) &= K^{-1} \otimes F + F \otimes 1, & \varepsilon(F) &= 0, & S(F) &= - K F, \\*
 \Delta(K) &= K \otimes K, & \varepsilon(K) &= 1, & S(K) &= K^{-1}.
\end{align*}
Lusztig actually considers the opposite coproduct, while we use the one
of Kassel \cite[Section~VII.1]{K95} and Majid \cite[Example~3.4.3]{M95}. The Hopf algebra $\bar{U}$ admits a ribbon structure \cite[Section~3.5]{R93} which is factorizable \cite[Corollary~A.3.3]{L94}, although we will need neither result in this paper. 
Moving on to the representation theory of $\bar{U}$, we recall the definition of simple and indecomposable projective $\bar{U}$-modules. The reader can find more details in \cite[Section~3]{S94}, although our notation is closer to the one of \cite[Appendix~C]{FGST05} and \cite[Section~4]{A08}. Note however that all these references actually focus on the restricted quantum group of $\fsl_2$, which corresponds to even values of $r$ instead of odd ones. For every integer $0 \leqs m \leqs r-1$ we denote with $X_m$ the simple $\bar{U}$-module with basis 
\[
 \{ a_j^m \mid 0 \leqs j \leqs m \}
\]
and action given, for all integers $0 \leqs j \leqs m$, by
\begin{equation*}
 \begin{split}
  K \cdot a_j^m &= q^{m-2j} a_j^m, \\
  E \cdot a_j^m &= [j][m-j+1] a_{j-1}^m, \\
  F \cdot a_j^m &= a_{j+1}^m,
 \end{split}
\end{equation*}
where $a_{-1}^m := a_{m+1}^m := 0$. As a $\bar{U}$-module, $X_m$ is generated by the \textit{highest weight vector} $a_0^m$. Next, for every integer $r \leqs m \leqs 2r-2$ we denote with $P_m$ the indecomposable projective $\bar{U}$-module with basis 
\[
 \{ a_j^m, x_k^m, y_k^m, b_j^m \mid 0 \leqs j \leqs 2r-m-2, 0 \leqs k \leqs m-r \}
\] 
and action given, for all integers $0 \leqs j \leqs 2r-m-2$ and $0 \leqs k \leqs m-r$, by
\begin{equation*}
 \begin{split}
  K \cdot a_j^m &= q^{-m-2j-2} a_j^m, \\
  E \cdot a_j^m &= -[j][m+j+1] a_{j-1}^m, \\
  F \cdot a_j^m &= a_{j+1}^m, \\
  K \cdot x_k^m &= q^{m-2k} x_k^m, \\
  E \cdot x_k^m &= [k][m-k+1] x_{k-1}^m, \\ 
  F \cdot x_k^m &= 
  \begin{cases}
   x_{k+1}^m & 0 \leqs k < m-r, \\
   a_0^m & k = m-r,
  \end{cases} \\
  K \cdot y_k^m &= q^{m-2k} y_k^m, \\
  E \cdot y_k^m &= 
  \begin{cases}
   a_{2r-m-2}^m & k = 0, \\
   [k][m-k+1] y_{k-1}^m & 0 < k \leqs m-r,
  \end{cases} \\
  F \cdot y_k^m &= y_{k+1}^m, \\
  K \cdot b_j^m &= q^{-m-2j-2} b_j^m, \\
  E \cdot b_j^m &= 
  \begin{cases}
   x_{m-r}^m & j = 0, \\
   a_{j-1}^m - [j][m+j+1] b_{j-1}^m & 0 < j \leqs 2r-m-2, 
  \end{cases} \\
  F \cdot b_j^m &= 
  \begin{cases}
   b_{j+1}^m & 0 \leqs j < 2r-m-2, \\
   y_0^m & j = 2r-m-2,
  \end{cases}
 \end{split}
\end{equation*}
where $a_{-1}^m := a_{2r-m-1}^m := x_{-1}^m := y_{m-r+1}^m := 0$. As a $\bar{U}$-module, $P_m$ is generated by the \textit{dominant vector} $b_0^m$. Let us mention that $P_m$ is usually denoted $P_{2r-m-2}$, because it is the projective cover of $X_{2r-m-2}$ for every integer $r \leqs m \leqs 2r-2$. Our change of notation is motivated by later convenience. 
Next, we need to specify identifications for relevant submodules of iterated tensor products of the fundamental simple $\bar{U}$-module $X := X_1$. Many details about decompositions of tensor products can be found in \cite[Section~3]{S94}, and an even more general treatment is given in \cite[Section~3]{KS09}. Again, these references discuss the restricted quantum group of $\fsl_2$, which corresponds to even values of $r$ rather than odd ones, but we will not actually need any of their results, as all the formulas we will list in here below can be easily checked by hand. For every integer $0 \leqs m \leqs r-1$, there is a unique isomorphic copy of $X_m$ among the submodules of $X^{\otimes m}$, which we still denote $X_m$ by abuse of notation. We construct it recursively. If $m = 0$, a standard basis of $X_0 \subset X^{\otimes 0}$ is obtained by setting $a_0^0 := 1$. If $0 < m \leqs r-1$, a standard basis of $X_m \subset X^{\otimes m}$ is obtained by setting
\begin{equation}\label{E:X_m}
 a_j^m := F^j \cdot \left( a_0^{m-1} \otimes a_0^1 \right).
\end{equation}
Similarly, for every integer $r \leqs m \leqs 2r-2$, there is an isomorphic copy of $P_m$ of multiplicity one among the submodules of $X^{\otimes m}$, which we still denote $P_m$ by abuse of notation. However, this submodule is no longer unique (for instance, although $X_{r-1} \otimes X_1$ and $X_1 \otimes X_{r-1}$ are isomorphic, they are not the same submodule of $X^{\otimes r}$). Therefore, we fix a choice recursively. If $m = r$, a standard basis of $P_r \subset X^{\otimes r}$ is obtained by setting
\begin{equation}\label{E:P_r}
 \begin{split}
  a_j^r &:= F^j \cdot \left( q a_0^{r-1} \otimes a_1^1 + a_1^{r-1} \otimes a_0^1 \right), \\
  x_0^r &:= a_0^{r-1} \otimes a_0^1, \\
  b_j^r &:= F^j \cdot \left( a_0^{r-1} \otimes a_1^1 \right), \\
  y_0^r &:= a_{r-1}^{r-1} \otimes a_1^1.
 \end{split}
\end{equation}
If $r < m \leqs 2r-2$, a standard basis of $P_m \subset X^{\otimes m}$ is obtained by setting
\begin{equation}\label{E:P_m}
 \begin{split}
  a_j^m &:= F^j \cdot \left( [m+1] q a_0^{m-1} \otimes a_1^1 + a_1^{m-1} \otimes a_0^1 \right), \\
  x_k^m &:= F^k \cdot \left( x_0^{m-1} \otimes a_0^1 \right), \\
  b_j^m &:= F^j \cdot \left( \frac{q^{-m}}{[m+1]} a_0^{m-1} \otimes a_1^1 + [m]q b_0^{m-1} \otimes a_1^1 + \frac{[m]}{[m+1]} b_1^{m-1} \otimes a_0^1 \right), \\
  y_k^m &:= F^k \cdot \left( \frac{q^{-m}}{[m+1]} a_{2r-m-1}^{m-1} \otimes a_1^1 + \frac{[m]}{[m+1]} y_0^{m-1} \otimes a_0^1 \right).
 \end{split}
\end{equation}
For $m = 2r+1$, we denote with $X_{r-1}^+$ and $X_{r-1}^-$ the two submodules of $X^{\otimes 2r-1}$ isomorphic to $X_{r-1}$ with standard bases obtained by setting
\begin{equation}\label{E:X_r-1_+_-}
 \begin{split}
  a_j^{r-1,+} &:= F^j \cdot \left( x_0^{2r-2} \otimes a_0^1 \right), \\
  a_j^{r-1,-} &:= F^j \cdot \left( q a_0^{2r-2} \otimes a_1^1 - y_0^{2r-2} \otimes a_0^1 \right).
 \end{split}
\end{equation}
The direct sum $X_{2r-1} := X_{r-1}^+ \oplus X_{r-1}^-$ can be lifted to a simple representation of the divided power quantum group, and it is this non-trivial splitting which will motivate the introduction of four additional generating morphisms in the Temperley--Lieb category. It will also be convenient to fix isomorphisms for supplements of these submodules. If $m = 2$, we denote with $X'_0$ the subspace of $X_1 \otimes X_1$ isomorphic to $X_0$ with standard basis obtained by setting
\begin{equation}\label{E:tilde_X_0}
 {a'}_0^0 := q a_0^1 \otimes a_1^1 - a_1^1 \otimes a_0^1.
\end{equation}
If $2 < m \leqs r-1$, we denote with $X'_{m-2}$ the subspace of $X_{m-1} \otimes X_1$ isomorphic to $X_{m-2}$ with standard basis obtained by setting
\begin{equation}\label{E:tilde_X_m-2}
 {a'}_j^{m-2} := F^j \cdot \left( [m-1]q a_0^{m-1} \otimes a_1^1 - a_1^{m-1} \otimes a_0^1 \right).
\end{equation}
If $m = r+1$, we denote with ${X'}_{r-1}^\up$ and ${X'}_{r-1}^\down$ the submodules of $X_{r-1} \otimes X_1 \otimes X_1$ isomorphic to $X_{r-1}$ with standard bases obtained by setting
\begin{equation}\label{E:tilde_X_r-1_+_-}
 \begin{split}
  {a'}_j^{r-1,\up} &:= F^j \cdot \left( q a_0^{r-1} \otimes a_0^1 \otimes a_1^1 - a_0^{r-1} \otimes a_1^1 \otimes a_0^1 \right), \\
  {a'}_j^{r-1,\down} &:= F^j \cdot \left( q a_0^{r-1} \otimes a_1^1 \otimes a_0^1 + a_1^{r-1} \otimes a_0^1 \otimes a_0^1 \right).
 \end{split}
\end{equation}
If $r \leqs m \leqs 2r-1$, we denote with $P'_{m-2}$ the subspace of $P_{m-1} \otimes X_1$ isomorphic to $P_{m-2}$  with standard basis obtained by setting
\begin{equation}\label{E:tilde_P_m-2}
 \begin{split}
  {a'}_j^{m-2} &:= F^j \cdot \left( -a_0^{m-1} \otimes a_0^1 \right), \\
  {x'}_k^{m-2} &:= F^k \cdot \left( [m-1]q x_0^{m-1} \otimes a_1^1 - x_1^{m-1} \otimes a_0^1 \right), \\ 
  {b'}_j^{m-2} &:= F^j \cdot \left( \frac{q^m}{[m-1]} x_{m-r-1}^{m-1} \otimes a_1^1 - \frac{[m]}{[m-1]} b_0^{m-1} \otimes a_0^1 \right), \\
  {y'}_k^{m-2} &:= F^k \cdot \left( [m]q y_0^{m-1} \otimes a_1^1 - \frac{[m]}{[m-1]} y_1^{m-1} \otimes a_0^1 \right).
 \end{split}
\end{equation}
Finally, let us fix our notation for $\bar{U}$-module morphisms. Dimensions of non-zero vector spaces of morphisms between simple and indecomposable projective objects of $\mods{\bar{U}}$ are given by
\begin{align*}
 \dim_\C \End_{\bar{U}}(X_m) &= 1 & 0 \leqs m \leqs r-1, \\*
 \dim_\C \End_{\bar{U}}(P_m) &= 2 & r \leqs m \leqs 2r-2, \\*
 \dim_\C \End_{\bar{U}}(X_{2r-1}) &= 4, & \\
 \dim_\C \Hom_{\bar{U}}(P_m,X_{2r-m-1}) &= \dim_\C \Hom_{\bar{U}}(X_{2r-m-1},P_m) = 1 & r \leqs m \leqs 2r-2, \\*
 \dim_\C \Hom_{\bar{U}}(P_m,P_{3r-m-2}) &= 2 & r \leqs m \leqs 2r-2, \\*
 \dim_\C \Hom_{\bar{U}}(X_{r-1},X_{2r-1}) &= \dim_\C \Hom_{\bar{U}}(X_{2r-1},X_{r-1}) = 2. &
\end{align*}
Non-trivial morphisms between simple and indecomposable projective objects of $\mods{\bar{U}}$ are:
\begin{itemize}[label=$\cdot$]
 \item $\varepsilon_m \in \End_{\bar{U}}(P_m)$ determined by
  \begin{equation}\label{E:epsilon_m_def}
   \varepsilon_m(a_j^m) := \varepsilon_m(x_k^m) := \varepsilon_m(y_k^m) := 0, \quad \varepsilon_m(b_j^m) := a_j^m
  \end{equation}
  for all integers $r \leqs m \leqs 2r-2$, $0 \leqs j \leqs 2r-m-2$, and $0 \leqs k \leqs m-r$;
 \item $\pi_m \in \Hom_{\bar{U}}(P_m,X_{2r-m-2})$ determined by
  \begin{equation}\label{E:pi_m_def}
   \pi_m(a_j^m) := \pi_m(x_k^m) := \pi_m(y_k^m) := 0, \quad \pi_m(b_j^m) := a_j^{2r-m-2}
  \end{equation}
  for all integers $r \leqs m \leqs 2r-2$, $0 \leqs j \leqs 2r-m-2$, and $0 \leqs k \leqs m-r$;
 \item $\iota_m \in \Hom_{\bar{U}}(X_{2r-m-2},P_m)$ determined by
  \begin{equation}\label{E:iota_m_def}
   \iota_m(a_j^{2r-m-2}) := a_j^m
  \end{equation}
  for all integers $r \leqs m \leqs 2r-2$ and $0 \leqs j \leqs 2r-m-2$;
 \item $\gamma_m^+ \in \Hom_{\bar{U}}(P_m,P_{3r-m-2})$ determined by
  \begin{equation}\label{E:gamma_m_+_def}
   \gamma_m^+(a_j^m) := \gamma_m^+(x_k^m) := 0, \quad \gamma_m^+(y_k^m) := a_k^{3r-m-2}, \quad \gamma_m^+(b_j^m) := x_j^{3r-m-2}
  \end{equation}
  for all integers $r \leqs m \leqs 2r-2$, $0 \leqs j \leqs 2r-m-2$, and $0 \leqs k \leqs m-r$;
 \item $\gamma_m^- \in \Hom_{\bar{U}}(P_m,P_{3r-m-2})$ determined by
  \begin{equation}\label{E:gamma_m_-_def}
   \gamma_m^-(a_j^m) := \gamma_m^-(y_k^m) := 0, \quad \gamma_m^-(x_k^m) := a_k^{3r-m-2}, \quad \gamma_m^-(b_j^m) := y_j^{3r-m-2}
  \end{equation}
  for all integers $r \leqs m \leqs 2r-2$, $0 \leqs j \leqs 2r-m-2$, and $0 \leqs k \leqs m-r$;
 \item $\pi_{2r-1}^\varepsilon \in \Hom_{\bar{U}}(X_{2r-1},X_{r-1})$ determined by
  \begin{equation}\label{E:pi_2r-1_def}
   \pi_{2r-1}^\varepsilon(a_j^{r-1,\varepsilon'}) := \delta_{\varepsilon,\varepsilon'} a_j^{r-1} 
  \end{equation}
  for all $\varepsilon,\varepsilon' \in \{ +,- \}$ and every integer $0 \leqs j \leqs r-1$;
 \item $\iota_{2r-1}^\varepsilon \in \Hom_{\bar{U}}(X_{r-1},X_{2r-1})$ determined by
  \begin{equation}\label{E:iota_2r-1_def}
   \iota_{2r-1}^\varepsilon(a_j^{r-1}) := a_j^{r-1,\varepsilon} 
  \end{equation}
  for every $\varepsilon \in \{ +,- \}$ and every integer $0 \leqs j \leqs r-1$.
\end{itemize}
These morphisms satisfy the following relations:
\begin{gather}
 \varepsilon_m = \iota_m \pi_m = \gamma_{3r-m-2}^+ \gamma_m^- = \gamma_{3r-m-2}^- \gamma_m^+, \label{E:epsilon_m_rel_1} \\
 \varepsilon_m^2 = \gamma_{3r-m-2}^+ \gamma_m^+ = \gamma_{3r-m-2}^- \gamma_m^- = 0, \label{E:epsilon_m_rel_2} \\
 \pi_{2r-1}^{\varepsilon'} \iota_{2r-1}^\varepsilon = \delta_{\varepsilon,\varepsilon'} \id_{X_{r-1}}, \label{E:pi_iota_rel_1} \\
 \sum_{\varepsilon,\varepsilon'} \iota_{2r-1}^\varepsilon \pi_{2r-1}^\varepsilon = \id_{X_{2r-1}}. \label{E:pi_iota_rel_2}
\end{gather}

\section{Temperley--Lieb category}\label{S:TL_category}

In this section we discuss a generalization of Jones--Wenzl idempotents at odd roots of unity. Non-semisimple Temperley--Lieb algebras at roots of unity and evaluable idempotents were first studied by Goodman and Wenzl in \cite{GW93}. Recursive formulas for these idempotents already appeared in the work of Ibanez \cite{I15} and Moore \cite{M18}, who independently treated the case of even roots of unity, which is complementary to ours. 

For an indeterminate $A$ and for every integer $k \geqs 0$ we introduce the notation 
\[
 \{ k \}_A := A^{2k} - A^{-2k},
 \quad \{ k \}'_A := A^{2k} + A^{-2k},
 \quad [k]_A := \frac{\{ k \}_A}{\{ 1 \}_A},
 \quad [k]_A! := \prod_{j=1}^k [j]_A.
\] 

The \textit{Temperley--Lieb category} $\TL_A$ is the $\C(A)$-linear category with set of objects given by $\N$, and with vector space of morphisms from $m \in \TL_A$ to $m' \in \TL_A$ denoted $\TL_A(m,m')$ and linearly generated by planar $(m,m')$-tangles modulo the subspace linearly generated by planar $(m,m')$-tangles of the form
\[
 u \cup t + [2]_A \cdot t,
\]
where $t$ is a planar $(m,m')$-tangle, and where $u$ is the unknot. Composition in $\TL_A$ is defined by vertical gluing, and denoted like multiplication, while tensor product is defined by horizontal juxtaposition. Equipped with this monoidal structure, it is easy to see that $\TL_A$ is a rigid category: for every $m \in \TL_A$ the dual $m^* \in \TL_A$ is given by $m$ itself, while for every $t \in \TL_A(m,m')$ the dual $t^* \in \TL_A(m',m)$ is obtained by applying a rotation of angle $\pi$ to $t$. For every $m \in \TL_A$, the $m$th \textit{Temperley--Lieb algebra} is the $\C(A)$-algebra $\TL_A(m) := \TL_A(m,m)$. 

Let us recall now the definition of \textit{simple Jones--Wenzl idempotents}. These endomorphisms of $\TL_A$ were first discovered by Jones \cite{J83}, although the recursive definition given here is due to Wenzl \cite{W87}, see also Lickorish \cite{L97}. For every integer $m \geqs 0$ the \textit{$m$th simple Jones--Wenzl idempotent} $f_m \in \TL_A(m)$ is recursively defined as 
\begin{equation}\label{E:f_def}
 \tl{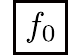} := \tl{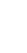} \qquad \tl{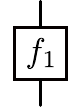} := \tl{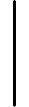} \qquad \tl{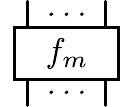} := \tl{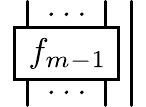} + \frac{[m-1]_A}{[m]_A} \cdot \hspace*{-5pt} \tl{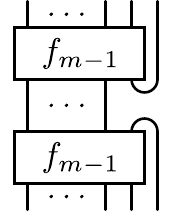}
\end{equation}
These endomorphisms satisfy
\begin{equation}\label{E:f_cups_caps}
 \tl{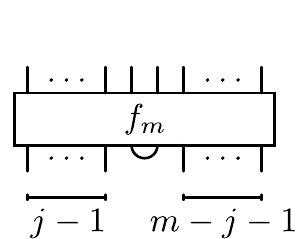} \hspace*{-5pt} = \tl{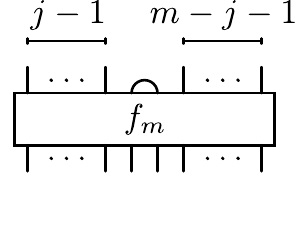} \hspace*{-5pt} = 0
\end{equation}
for all integers $1 \leqs j \leqs m-1$. This implies
\begin{equation}\label{E:f_f}
 \tl{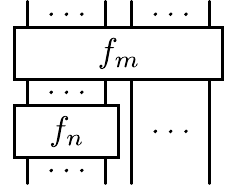} = \tl{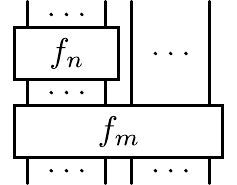} = \tl{f_m}
\end{equation}
for all integers $0 \leqs n \leqs m$. It can be shown that $f^*_m = f_m$ for every integer $m \geqs 0$. Furthermore, for all integers $0 \leqs k \leqs m$ we have
\begin{equation}\label{E:f_ptr_k}
 \tl{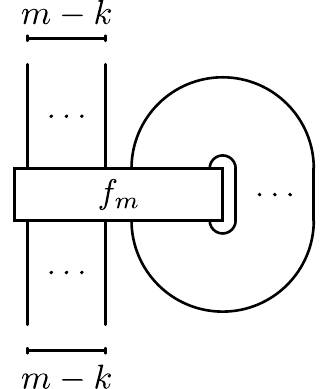} = (-1)^{k} \frac{[m+1]_A}{[m-k+1]_A} \cdot \tl{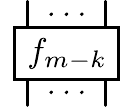}
\end{equation}
Now we can define a new family of endomorphisms $g_m, h_m \in \TL_A(m)$ for every integer $r \leqs m \leqs 2r-2$ by setting 
\begin{equation}
 g_m := f_m + \frac{1}{[r]_A} \cdot h_m \label{E:g_def}
\end{equation}
where
\begin{equation}
 \tl{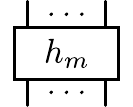} := (-1)^m [2r-m-1]_A \cdot \hspace*{-5pt} \tl{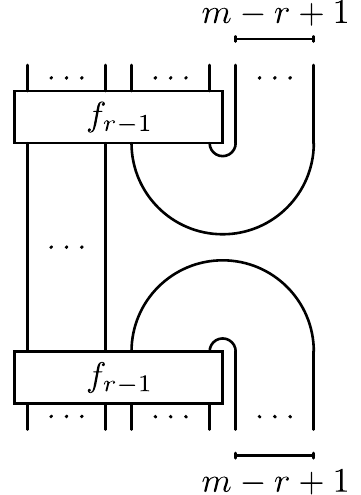} \label{E:h_def}
\end{equation}
These endomorphisms satisfy
\begin{align}
 \tl{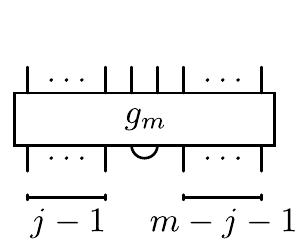} \hspace*{-5pt} &= \tl{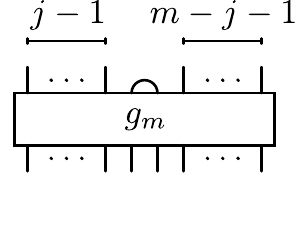} \hspace*{-5pt} = 0 \qquad \qquad j \neq r-1 \label{E:g_cups_caps} \\
 \tl{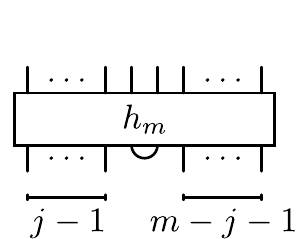} \hspace*{-5pt} &= \tl{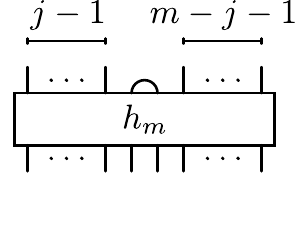} \hspace*{-5pt} = 0 \qquad \qquad j \neq r-1 \label{E:h_cups_caps}
\end{align}
for all integers $r \leqs m \leqs 2r-2$ and $1 \leqs j \leqs m-1$, provided $j$ is not equal to $r-1$. We clearly have
\begin{align}
 \tl{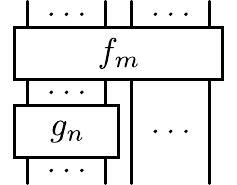} &= \tl{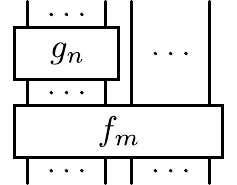} = \tl{f_m} \label{E:f_g} \\
 \tl{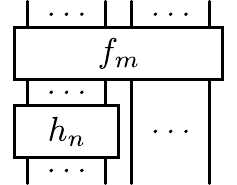} &= \tl{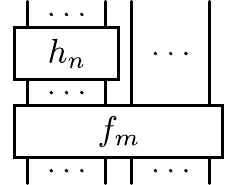} = 0 \label{E:f_h}
\end{align}
for all integers $r \leqs n \leqs m \leqs 2r-2$. Furthermore, a direct computation involving equation~\eqref{E:f_ptr_k} shows
\begin{align}
 \tl{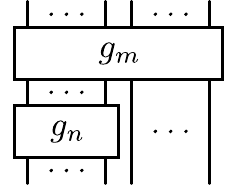} &= \tl{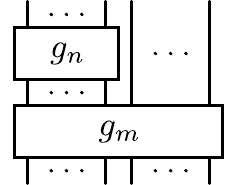} = \tl{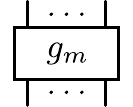} \label{E:g_g} \\
 \tl{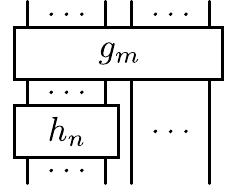} &= \tl{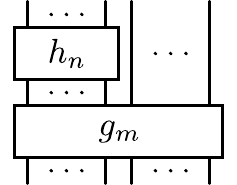} = \tl{h_m} \label{E:g_h} \\
 \tl{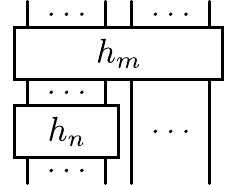} &= \tl{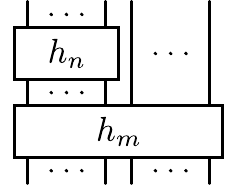} = [r]_A \cdot \tl{h_m} \label{E:h_h} 
\end{align}
for all integers $r \leqs n \leqs m \leqs 2r-2$. Observe that $g^*_{2r-2} = g_{2r-2}$ and $h^*_{2r-2} = h_{2r-2}$.

\begin{lemma}\label{L:g_h_ptr_k}
 For every integer $r \leqs m \leqs 2r-2$, if $0 \leqs k \leqs m-r$ we have
 \begin{align}
  \tl{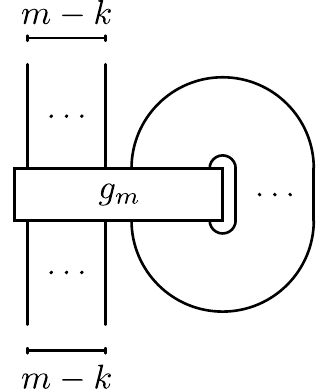} &=
  (-1)^k \frac{[m+1]_A}{[m-k+1]_A} \cdot \tl{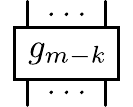} \nonumber \\*
  &\hspace*{\parindent} - (-1)^k \frac{\{ r \}'_A [k]_A}{[2r-m+k-1]_A [m-k+1]_A} \cdot \tl{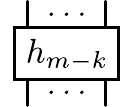} \label{E:g_ptr_k_a} \\
   \tl{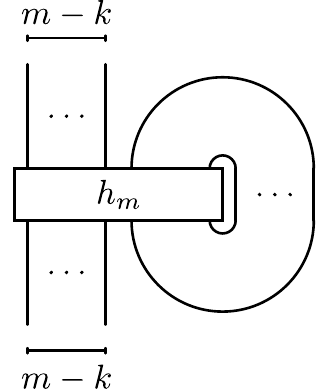} &=
   (-1)^k \frac{[2r-m-1]_A}{[2r-m+k-1]_A} \cdot \tl{h_m-k} \label{E:h_ptr_k_a}
 \end{align}
 while if $m-r < k \leqs m$ we have
 \begin{align}
  \tl{g_m_ptr_k} &= (-1)^k \frac{[ r ]_A \{ m-r+1 \}'_A}{[ m-k+1 ]_A} \cdot \tl{f_m-k} \label{E:g_ptr_k_b} \\
  \tl{h_m_ptr_k} &=
  (-1)^k \frac{[r]_A [2r-m-1]_A}{[m-k+1]_A} \cdot \tl{f_m-k} \label{E:h_ptr_k_b}
 \end{align}
\end{lemma}

\begin{proof}
 Equation~\eqref{E:h_ptr_k_a} follows directly from equation~\eqref{E:h_def}, and this implies equation~\eqref{E:g_ptr_k_a} using equation~\eqref{E:f_ptr_k} together with the equality
 \[
  \frac{\{ r \}'_A[k]_A}{[2r-m+k-1]_A[m-k+1]_A} = \frac{1}{[r]_A} \left(\frac{[m+1]_A}{[m-k+1]_A} - \frac{[2r-m-1]_A}{[2r-m+k-1]_A} \right).
 \]
 Similarly, equation~\eqref{E:h_ptr_k_b} follows directly from equation~\eqref{E:f_ptr_k}, and this implies equation~\eqref{E:g_ptr_k_b} using equation~\eqref{E:f_ptr_k} together with the equality
 \[
  [r]_A \{ m-r+1 \}'_A = [m+1]_A + [2r-m-1]_A. \qedhere
 \]
\end{proof}

\begin{lemma}\label{L:recurrence_g_h}
 The idempotent $g_m$ satisfies
 \begin{align}
  \tl{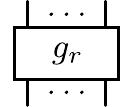} &= \tl{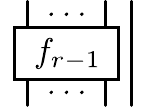} \label{E:g_r} \\
   \tl{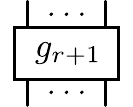} &= \tl{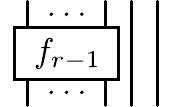} + \frac{[r-1]_A}{[r+1]_A} \left( \tl{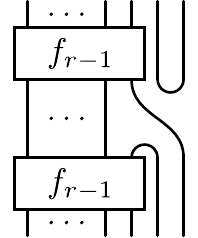} + \tl{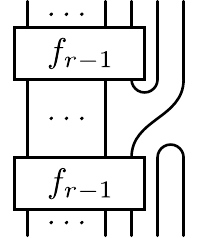} \right) \nonumber \\*
   &\hspace*{\parindent} + \frac{[r-1]_A [2]_A}{[r+1]_A} \cdot \tl{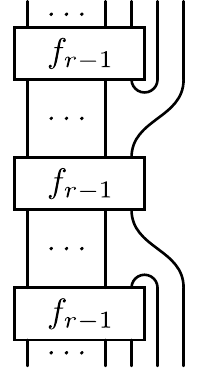} + \frac{[r]_A}{[r+1]_A} \cdot \tl{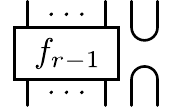} \label{E:g_r+1} \\
   \tl{g_m} &= \tl{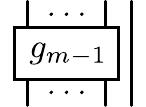} + \frac{[m-1]_A}{[m]_A} \cdot \hspace*{-5pt} \tl{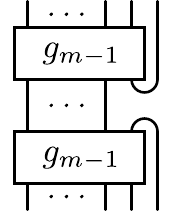} + \frac{\{ r \}'_A}{[2r-m]_A [m]_A} \cdot \tl{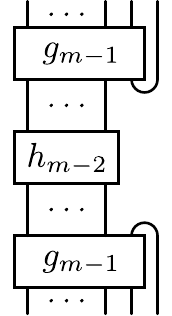} \hspace*{-5pt} \label{E:g_m}
 \end{align}
 for every integer $r+2 \leqs m \leqs 2r-2$.
\end{lemma}

\begin{proof}
 The statement is proved by direct computation. For $m = r$, equation~\eqref{E:g_r} follows from the recurrence formula~\eqref{E:f_def} defining $f_r$. For $m = r+1$, equation~\eqref{E:g_r+1} follows from a double application of equation~\eqref{E:f_def}. For $r+1 < m \leqs 2r-2$, let us look at the right-hand side of equation~\eqref{E:g_m}. For what concerns the first term, equation~\eqref{E:g_def} yields
 \[
  \tl{g_m_a} = \tl{f_m_a} + \frac{1}{[r]_A} \cdot \tl{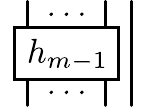}
 \]
 For what concerns the second term, equations~\eqref{E:f_h} and \eqref{E:g_h} yield
 \begin{align*}
  \tl{g_m_b} 
  &= \tl{f_m_b} \hspace*{-5pt} + \frac{1}{[r]_A} \cdot \hspace*{-5pt} \tl{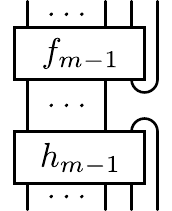} \hspace*{-5pt} + \frac{1}{[r]_A} \cdot \hspace*{-5pt} \tl{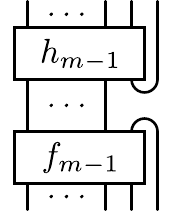} \hspace*{-5pt} + \frac{1}{[r]_A^2} \cdot \hspace*{-5pt} \tl{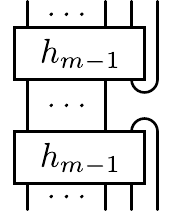} \\*  
  &= \tl{f_m_b} \hspace*{-5pt} + \frac{1}{[r]_A^2} \cdot \hspace*{-5pt} \tl{h_m_b}
 \end{align*}
 For what concerns the third term, equations~\eqref{E:g_h} and \eqref{E:h_h} yield
 \[
  \tl{g_m_c} = \frac{1}{[r]_A} \cdot \tl{h_m_b}
 \]
 This means that, using the equality
 \[
  \frac{\{ r \}'_A}{[2r-m]_A[m]_A} = \frac{1}{[r]_A} \left(\frac{[2r-m+1]_A}{[2r-m]_A} - \frac{[m-1]_A}{[m]_A} \right),
 \]
 we can rewrite the right-hand side of equation~\eqref{E:g_m} as
 \[
  \tl{f_m_a} + \frac{[m-1]_A}{[m]_A} \cdot \hspace*{-5pt} \tl{f_m_b} + \frac{1}{[r]_A} \cdot \tl{h_m_a} + \frac{[2r-m+1]_A}{[r]_A^2 [2r-m]_A} \cdot \hspace*{-5pt} \tl{h_m_b}
 \]
 But now we have
 \begin{align*}
  \tl{h_m_b} 
  &= [2r-m]_A^2 \cdot \tl{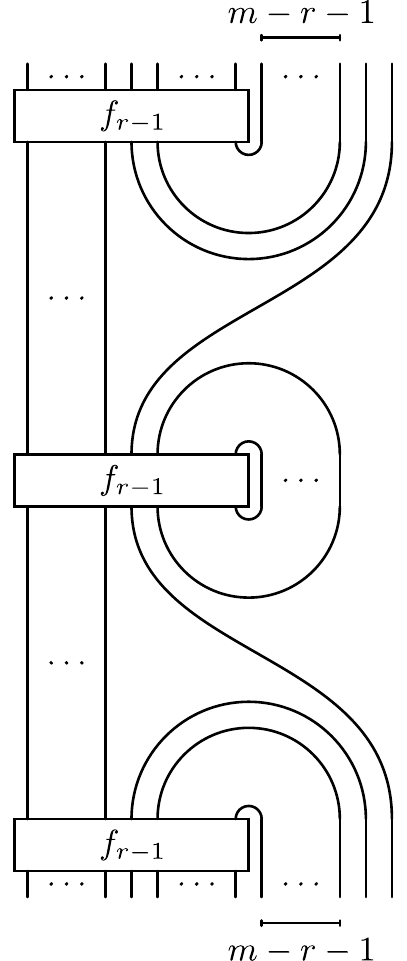} \\*
  &= (-1)^m \frac{[r]_A [2r-m]_A^2}{[2r-m+1]_A} \cdot \tl{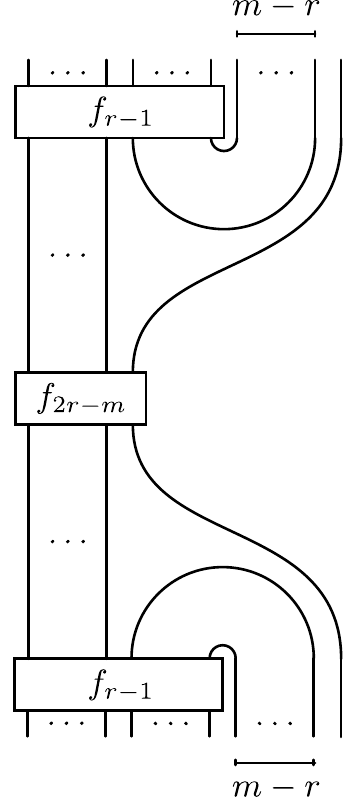} \\*
  &= \frac{[r]_A [2r-m]_A}{[2r-m+1]_A} \left( - \tl{h_m_a} + \tl{h_m} \right) 
 \end{align*}
 where the first equality follows from equation~\eqref{E:h_def}, the second equality follows from equation~\eqref{E:f_ptr_k}, and the third equality follows from equation~\eqref{E:f_def}. This means that the right-hand side of equation~\eqref{E:g_m} can be further rewritten as
 \[
  \tl{f_m_a} + \frac{[m-1]_A}{[m]_A} \cdot \hspace*{-5pt} \tl{f_m_b} + \frac{1}{[r]_A} \cdot \tl{h_m} = \tl{f_m} + \frac{1}{[r]_A} \cdot \tl{h_m}
 \]
 Now the statement follows from equation~\eqref{E:g_def}. 
\end{proof}

\begin{lemma}\label{L:f_2r-1}
 The idempotent $f_{2r-1}$ satisfies
 \begin{equation}
   \tl{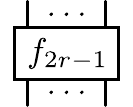} = \tl{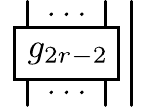} + \frac{[2r-2]_A}{[2r-1]_A} \cdot \tl{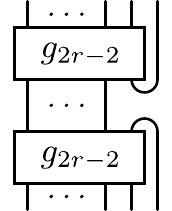} + \frac{\{ r \}'_A}{[2r-1]_A} \cdot \tl{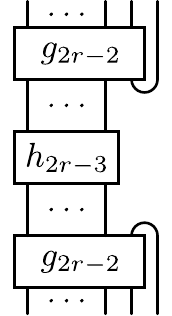} \label{E:f_2r-1} 
 \end{equation}
\end{lemma}

\begin{proof}
 We follow the proof of equation~\eqref{E:g_m} in Lemma~\ref{L:recurrence_g_h}, but this time we suppose $m = 2r-1$. Then, the same argument implies that the right-hand side of equation~\eqref{E:f_2r-1} can be rewritten as
 \[
  \tl{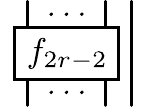} + \frac{[2r-2]_A}{[2r-1]_A} \cdot \hspace*{-5pt} \tl{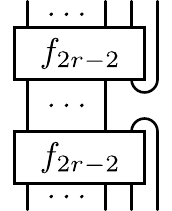} + \frac{1}{[r]_A} \cdot \tl{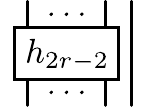} + \frac{[2]_A}{[r]_A^2} \cdot \hspace*{-5pt} \tl{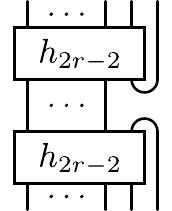}
 \]
 But now we have
 \begin{align*}
  \tl{h_h_ref} 
  &= \tl{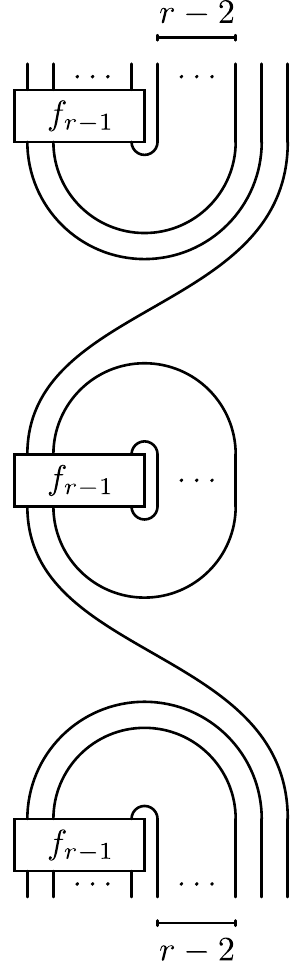} 
  = -\frac{[r]_A}{[2]_A} \cdot \hspace*{-15pt} \tl{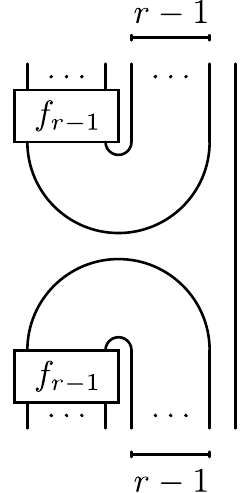} 
  = -\frac{[r]_A}{[2]_A} \cdot \tl{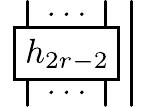}
 \end{align*}
 where the first equality follows from equation~\eqref{E:h_def}, the second equality follows from equation~\eqref{E:f_ptr_k}, and the third equality follows from equation~\eqref{E:f_def}. Now the statement follows from equation~\eqref{E:f_def}. 
\end{proof}

From now on, we will focus on the specialization
\[
 \TL := \TL_{q^{\frac{r+1}{2}}}.
\]
We observe that $f_m$ can be specialized to $\TL$ only for $0 \leqs m \leqs r-1$, because the formula defining $f_r$ has a pole at $A = q^{\frac{r+1}{2}}$. On the other hand, thanks to Lemma~\ref{L:recurrence_g_h}, $g_m$ can be specialized to $\TL$ for all $r \leqs m \leqs 2r-2$, and similarly, thanks to Lemma~\ref{L:f_2r-1}, $f_{2r-1}$ can be specialized to $\TL$. We call $g_m$ the \textit{$m$th non-semisimple Jones--Wenzl idempotent}.

\begin{remark}\label{R:specialization}
 It can be proven by induction that in $\TL$ the idempotent $f_{r-1}$ satisfies
 \begin{equation}
  \tl{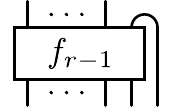} = - \tl{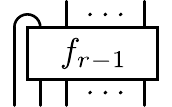} \qquad \qquad
  \tl{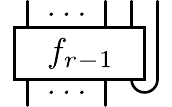} = - \tl{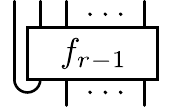} \label{E:f_ambi}
 \end{equation}
 In particular, in $\TL$ we have $h^*_m = h_m$ for every integer $r \leqs m \leqs 2r-2$, because a rotation of angle $\pi$ fixes $h_m$. Remark however that this is true only once we specialize $A$ to $q^{\frac{r+1}{2}}$, so the same property does not hold for $g_m$.
\end{remark}

\section{Monoidal linear functor}\label{S:monoidal_linear_functor}

In this section, we study a monoidal linear functor $F_{\TL} : \TL \to \mods{\bar{U}}$, compare with \cite{FK97}. By abuse of notation, we still denote by $\TL$ the idempotent completion of $\TL$. This means we promote idempotent endomorphisms $p_m \in \TL(m)$ to objects of $\TL$, and we set
\[
 \TL(p_m,p_{m'}) = \{ t \in \TL(m,m') \mid t p_m = t = p_{m'} t \},
\]
with the morphism $p_m$ being the identity of the object $p_m$. As a consequence, we will occasionally confuse direct summands of objects of $\mods{\bar{U}}$ with the corresponding idempotent endomorphisms. We also start adopting a graphical notation featuring labels given by integers $m \geqs 0$ placed next to endpoints of edges, standing for the number of parallel strands contained in the plane.

Let us consider the monoidal linear functor
\[ 
 F_{\TL} : \TL \to \mods{\bar{U}}
\]
sending the object $1 \in \TL$ to $X \in \mods{\bar{U}}$, and sending the morphism $\cup \in \TL(0,2)$ to $c \in \Hom_{\bar{U}}(\C,X \otimes X)$ defined by
\[
 c(1) := q a_0^1 \otimes a_1^1 - a_1^1 \otimes a_0^1,
\]
and the morphism $\cap \in \TL(2,0)$ to $e \in \Hom_{\bar{U}}(X \otimes X,\C)$ defined by
\[
 e(a_0^1 \otimes a_0^1) = 0, \quad 
 e(a_0^1 \otimes a_1^1) = -1, \quad
 e(a_1^1 \otimes a_0^1) = q^{-1}, \quad 
 e(a_1^1 \otimes a_1^1) = 0.
\]
In order to study $F_{\TL}$, let us define projection morphisms $p_m \in \TL(g_m,f_{2r-m-2})$ and injection morphisms $i_m \in \TL(f_{2r-m-2},g_m)$ as
\begin{equation}
 \tl{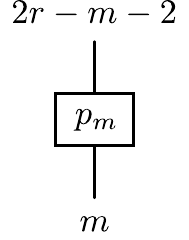} \hspace*{-7.5pt} := \tl{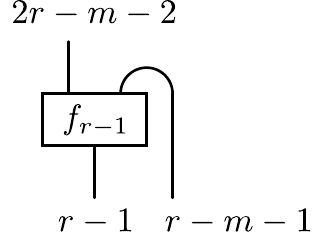} \qquad
 \tl{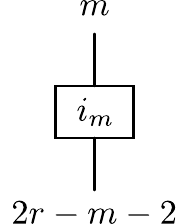} \hspace*{-7.5pt} := \tl{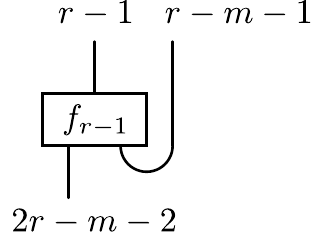} \hspace*{-7.5pt} \label{E:p_i_def}
\end{equation}
for every integer $r \leqs m \leqs 2r-2$, which immediately gives
\begin{equation}
 \tl{p_m} \hspace*{-7.5pt} = \tl{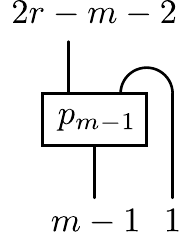} \qquad
 \tl{i_m} \hspace*{-7.5pt} = \tl{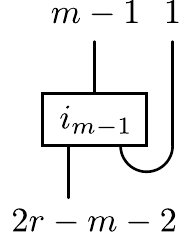} \label{E:p_i_m}
\end{equation}
The fact that $p_m g_m = p_m$ and that $g_m i_m = i_m$ follows from a direct computation.

\begin{lemma}\label{L:functor}
 For every integer $0 \leqs m \leqs r-1$ the object $f_m$ satisfies
 \begin{equation}
  F_{\TL}(f_m) = X_m, \label{E:F(f_m)} 
 \end{equation}
 for $r \leqs m \leqs 2r-2$ the object $g_m$ and the morphisms $h_m$, $p_m$, and $i_m$ satisfy
 \begin{align}
  F_{\TL}(g_m) &= P_m, \label{E:F(g_m)} \\*
  F_{\TL}(h_m) &= -\frac{1}{[m+1]} \varepsilon_m, \label{E:F(h_m)} \\*
  F_{\TL}(p_m) &= (-1)^m \frac{[m-r]!}{[m+1]} \pi_m, \label{E:F(p_m)} \\*
  F_{\TL}(i_m) &= \frac{1}{[m-r+1]!} \iota_m, \label{E:F(i_m)}
 \end{align}
 and for $m = 2r-1$ the object $f_{2r-1}$ satisfies
 \begin{equation}
  F_{\TL}(f_{2r-1}) = X_{2r-1}, \label{E:F(f_2r-1)} 
 \end{equation}
 where the morphisms $\varepsilon_m$, $\pi_m$, and $\iota_m$ are defined by equations~\eqref{E:epsilon_m_def} to \eqref{E:iota_m_def}.
\end{lemma}

The proof of Lemma~\ref{L:functor} is a lengthy computation which will occupy the remainder of this section.

\begin{proof}
 In order to discuss the proof, we first need to fix some notation. Let us denote with $Y_m$ the unique submodule of $X^{\otimes m}$ satisfying 
 \begin{equation}
  X^{\otimes m} \cong X_m \oplus Y_m \label{E:decomposition_simple}
 \end{equation}
 for every integer $0 \leqs m \leqs r-1$, and let us denote with $\varphi_m \in \End_{\bar{U}}(X^{\otimes m})$ the idempotent endomorphism with image $X_m$ and kernel $Y_m$. Next, let us denote with $Q_m$ the unique submodule of $X^{\otimes m}$ satisfying 
 \begin{equation}
  X^{\otimes m} \cong P_m \oplus Q_m \label{E:decomposition_projective}
 \end{equation}
 for every integer $r \leqs m \leqs 2r-2$, and let us denote with $\psi_m \in \End_{\bar{U}}(X^{\otimes m})$ the idempotent endomorphism with image $P_m$ and kernel $Q_m$. Similarly, let us denote with $Y_{2r-1}$ the unique submodule of $X^{\otimes 2r-1}$ satisfying 
 \begin{equation}
  X^{\otimes 2r-1} \cong X_{r-1}^+ \oplus X_{r-1}^- \oplus Y_{2r-1}, \label{E:decomposition_2r-1}
 \end{equation}
 and for every $\varepsilon \in \{ +,- \}$ let us denote with $\varphi_{2r-1}^\varepsilon \in \End_{\bar{U}}(X^{\otimes 2r-1})$ the idempotent endomorphism with image $X_{2r-1}^\varepsilon$ and kernel $X_{r-1}^{\overline{\varepsilon}} \oplus Y_{2r-1}$, where $\overline{\varepsilon} \in \{ +,- \}$ denotes the opposite of $\varepsilon$. 
 
 The strategy will be to prove by induction on $0 \leqs m \leqs 2r-1$ that $F_{\TL}(f_m) = \varphi_m$, if $0 \leqs m \leqs r-1$, that $F_{\TL}(g_m) = \psi_m$, if $r \leqs m \leqs 2r-2$, and finally that $F_{\TL}(f_{2r-1}) = \varphi_{2r-1}^+ + \varphi_{2r-1}^-$. This is precisely what equations~\eqref{E:F(f_m)}, \eqref{E:F(g_m)}, and \eqref{E:F(f_2r-1)} stand for. This will allow us to establish equations~\eqref{E:F(p_m)} and \eqref{E:F(i_m)}, which immediately imply equation~\eqref{E:F(h_m)}, by restriction to the submodule $P_m$ of $X^{\otimes m}$. An important observation is that, since the functor $F_{\TL}$ sends all morphisms if $\TL$ to intertwiners, it is sufficient to check these equalities only for highest-weight vectors, in the case of simple $\bar{U}$-modules, and dominant vectors, in the case of indecomposable projective $\bar{U}$-modules.
 
 For $m = 0$ and for $m = 1$, equation~\eqref{E:F(f_m)} follows directly from the definitions. For $1 < m \leqs r-1$, specializing equation~\eqref{E:f_def} at $A = q^{\frac{r+1}{2}}$ gives
 \[
  f_m = f_{m-1} \otimes f_1 - f'_{m-2},
 \]
 where
 \begin{equation}
  \fusion{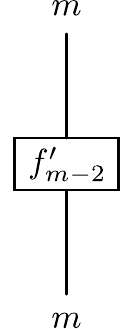} := -\frac{[m-1]}{[m]} \cdot \fusion{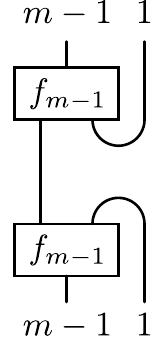} \label{E:f_prime_m-2}
 \end{equation}
 We know $F_{\TL}(f_{m-1} \otimes f_1)$ restricts to the identity on $X_{m-1} \otimes X_1$ and to zero on $Y_{m-1} \otimes X_1$ thanks to the induction hypothesis. Then, using equation~\eqref{E:X_m}, we have
 \[
   a_0^m = a_0^{m-1} \otimes a_0^1 = a_0^{m-2} \otimes a_0^1 \otimes a_0^1.
 \]
 This means that, using the definition of $F_\TL$ and equation~\eqref{E:f_prime_m-2}, we get
 \begin{align*}
  F_{\TL}(f'_{m-2})(a_0^m) &= 0, & F_{\TL}(f_m)(a_0^m) &= a_0^m.
 \end{align*}
 On the other hand, using equation~\eqref{E:tilde_X_m-2}, we have
 \begin{align*}
  \begin{split}
   {a'}_0^{m-2} &= [m-1]q a_0^{m-1} \otimes a_1^1 - a_1^{m-1} \otimes a_0^1 \\
   &= [m-1]q a_0^{m-2} \otimes a_0^1 \otimes a_1^1 - q^{-m+2} a_0^{m-2} \otimes a_1^1 \otimes a_0^1 - a_1^{m-2} \otimes a_0^1 \otimes a_0^1,
  \end{split}
 \end{align*}
 and, using equations~\eqref{E:X_m} and \eqref{E:tilde_X_m-2}, we have
 \begin{align*}
  a_0^{m-2} \otimes a_0^1 &= a_0^{m-1}, &
  a_0^{m-2} \otimes a_1^1 &= \frac{1}{[m-1]} a_1^{m-1} + \frac{1}{[m-1]} {a'}_0^{m-3}.
 \end{align*}
 This means that, using the definition of $F_\TL$ and equation~\eqref{E:f_prime_m-2}, we get
 \begin{align*}
  F_{\TL}(f'_{m-2})({a'}_0^{m-2}) &= {a'}_0^{m-2}, &
  F_{\TL}(f_m)({a'}_0^{m-2}) &= 0.
 \end{align*}
 Therefore, $F_{\TL}(f_m) = \varphi_m$. This proves equation~\eqref{E:F(f_m)}. For $m = r$, equation~\eqref{E:g_r} gives
 \[
  g_r = f_{r-1} \otimes f_1.
 \]
 We know $F_{\TL}(f_{r-1} \otimes f_1)$ restricts to the identity on $X_{r-1} \otimes X_1$ and to zero on $Y_{r-1} \otimes X_1$ thanks to equation~\eqref{E:F(f_m)}, and so $F_{\TL}(g_r) = \psi_r$. This proves equation~\eqref{E:F(g_m)}. For what concerns equation~\eqref{E:F(p_m)}, using equation~\eqref{E:P_r}, we have
 \[
  b_0^r = a_0^{r-1} \otimes a_1^1 = a_0^{r-2} \otimes a_0^1 \otimes a_1^1.
 \]
 Then, using the definition of $F_\TL$ and equation~\eqref{E:p_i_m}, this means
 \[
  F_{\TL}(p_r)(b_0^r) = -a_0^{r-2}.
 \]
 For what concerns equation~\eqref{E:F(i_m)}, using equations~\eqref{E:X_m} and \eqref{E:tilde_X_m-2}, we have
 \begin{align*}
  a_0^{r-2} \otimes a_0^1 &= a_0^{r-1}, &
  a_0^{r-2} \otimes a_1^1 &= - a_1^{r-1} - {a'}_0^{r-3},
 \end{align*}
 and, using equation~\eqref{E:P_r}, we have
 \[
  a_0^r = q a_0^{r-1} \otimes a_1^1 + a_1^{r-1} \otimes a_0^1.
 \]
 Then, using the definition of $F_\TL$ and equation~\eqref{E:p_i_m}, this means
 \[
  F_{\TL}(i_r)(a_0^{r-2}) = a_0^r.
 \]
 Now equation~\eqref{E:F(h_m)} follows from equations~\eqref{E:F(p_m)} and \eqref{E:F(i_m)}. For $m = r+1$, specializing equation~\eqref{E:g_r+1} at $A = q^{\frac{r+1}{2}}$ gives
 \[
  g_{r+1} = f_{r-1} \otimes f_1 \otimes f_1 - {f'}_{r-1}^\up - {f'}_{r-1}^\down
 \]
 where
 \begin{equation}
  \fusion{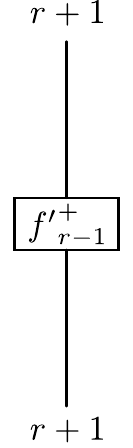} := \fusion{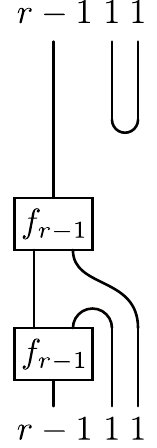} \qquad \qquad
  \fusion{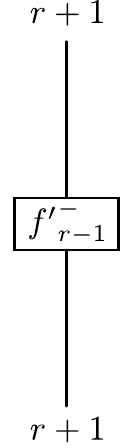} := \fusion{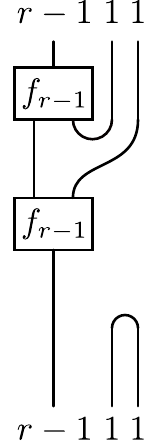} + [2] \cdot \fusion{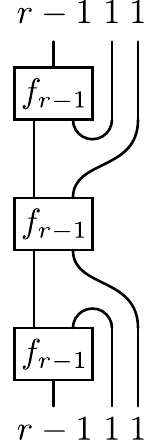} \label{E:f_prime_r-1}
 \end{equation}
 We know $F_{\TL}(f_{r-1} \otimes f_1 \otimes f_1)$ restricts to the identity on $X_{r-1} \otimes X_1 \otimes X_1$ and to zero on $Y_{r-1} \otimes X_1 \otimes X_1$. Then, using equation~\eqref{E:P_m}, we have
 \begin{align*}
  \begin{split}
   b_0^{r+1} &= \left( \frac{1}{[2]} + q \right) a_0^{r-1} \otimes a_1^1 \otimes a_1^1 + \frac{q^{-1}}{[2]} a_1^{r-1} \otimes a_0^1 \otimes a_1^1 + \frac{1}{[2]} a_1^{r-1} \otimes a_1^1 \otimes a_0^1 \\
   &= \left( \frac{1}{[2]} + q \right) a_0^{r-2} \otimes a_0^1 \otimes a_1^1 \otimes a_1^1 + \frac{q}{[2]} a_0^{r-2} \otimes a_1^1 \otimes a_0^1 \otimes a_1^1 \\
   &\hspace*{\parindent} + \frac{q^2}{[2]} a_0^{r-2} \otimes a_1^1 \otimes a_1^1 \otimes a_0^1 + \frac{q^{-1}}{[2]} a_1^{r-2} \otimes a_0^1 \otimes a_0^1 \otimes a_1^1 \\
   &\hspace*{\parindent} + \frac{1}{[2]} a_1^{r-2} \otimes a_0^1 \otimes a_1^1 \otimes a_0^1,
  \end{split}
 \end{align*}
 and, using equation~\eqref{E:tilde_X_m-2}, we have
 \begin{align*}
  a_0^{r-2} \otimes a_1^1 &= - a_1^{r-1} - {a'}_0^{r-3}, &
  a_1^{r-2} \otimes a_0^1 &= [2]q a_1^{r-1} + q^2 {a'}_0^{r-3}.
 \end{align*}  
 This means that, using the definition of $F_\TL$ and equation~\eqref{E:f_prime_r-1}, we get
 \begin{align*}
  F_{\TL}({f'}_{r-1}^\up)(b_0^{r+1}) &= 0, & 
  F_{\TL}({f'}_{r-1}^\down)(b_0^{r+1}) &= 0, & 
  F_{\TL}(g_{r+1})(b_0^{r+1}) &= b_0^{r+1}.
 \end{align*}
 On the other hand, using equation~\eqref{E:tilde_X_r-1_+_-}, we have
 \begin{align*}
  \begin{split}
   {a'}_0^{r-1,\up} &= q a_0^{r-1} \otimes a_0^1 \otimes a_1^1 - a_0^{r-1} \otimes a_1^1 \otimes a_0^1 \\
   &= q a_0^{r-2} \otimes a_0^1 \otimes a_0^1 \otimes a_1^1 - a_0^{r-2} \otimes a_0^1 \otimes a_1^1 \otimes a_0^1,
  \end{split} \\
  \begin{split}
   {a'}_0^{r-1,\down} &= q a_0^{r-1} \otimes a_1^1 \otimes a_0^1 + a_1^{r-1} \otimes a_0^1 \otimes a_0^1 \\
   &= q a_0^{r-2} \otimes a_0^1 \otimes a_1^1 \otimes a_0^1 + q^2 a_0^{r-2} \otimes a_1^1 \otimes a_0^1 \otimes a_0^1 + a_1^{r-2} \otimes a_0^1 \otimes a_0^1 \otimes a_0^1.
  \end{split}  
 \end{align*}
 This means that, using the definition of $F_\TL$ and equation~\eqref{E:f_prime_r-1}, we get
 \begin{align*}
  F_{\TL}({f'}_{r-1}^\up)({a'}_0^{r-1,\up}) &= {a'}_0^{r-1,\up}, &
  F_{\TL}({f'}_{r-1}^\up)({a'}_0^{r-1,\down}) &= 0, \\*
  F_{\TL}({f'}_{r-1}^\down)({a'}_0^{r-1,\up}) &= 0, &
  F_{\TL}({f'}_{r-1}^\down)({a'}_0^{r-1,\down}) &= {a'}_0^{r-1,\down}, \\*
  F_{\TL}(g_{r+1})({a'}_0^{r-1,\up}) &= 0, &
  F_{\TL}(g_{r+1})({a'}_0^{r-1,\down}) &= 0.
 \end{align*}
 Therefore, $F_{\TL}(g_{r+1}) = \psi_{r+1}$. This proves equation~\eqref{E:F(g_m)}. For what concerns equation~\eqref{E:F(p_m)}, using equation~\eqref{E:P_m}, we have
 \[
  b_0^{r+1} = \frac{q^{-1}}{[2]} a_0^r \otimes a_1^1 + q b_0^r \otimes a_1^1 + \frac{1}{[2]} b_1^r \otimes a_0^1,
 \]
 and, using equation~\eqref{E:X_m}, we have
 \begin{align*}
  a_0^{r-2} &= a_0^{r-3} \otimes a_0^1, &
  a_1^{r-2} &= q^3 a_0^{r-3} \otimes a_1^1 + a_1^{r-3} \otimes a_0^1.
 \end{align*}
 Then, using the definition of $F_\TL$ and equation~\eqref{E:p_i_m}, this means
 \[
  F_{\TL}(p_{r+1})(b_0^{r+1}) = \frac{1}{[2]} a_0^{r-3}.
 \]
 For what concerns equation~\eqref{E:F(i_m)}, using equations~\eqref{E:X_m} and \eqref{E:tilde_X_m-2}, we have
 \begin{align*}
  a_0^{r-3} \otimes a_0^1 &= a_0^{r-2}, &
  a_0^{r-3} \otimes a_1^1 &= -\frac{1}{[2]} a_1^{r-2} -\frac{1}{[2]} {a'}_0^{r-4},
 \end{align*}
 and, using equation~\eqref{E:P_m}, we have
 \[
  a_0^{r+1} = [2]q a_0^r \otimes a_1^1 + a_1^r \otimes a_0^1.
 \]
 Then, using the definition of $F_\TL$ and equation~\eqref{E:p_i_m}, this means
 \[
  F_{\TL}(i_{r+1})(a_0^{r-3}) = \frac{1}{[2]} a_0^{r+1}.
 \]
 Now equation~\eqref{E:F(h_m)} follows from equations~\eqref{E:F(p_m)} and \eqref{E:F(i_m)}. For $r+1 < m \leqs 2r-2$, specializing equation~\eqref{E:g_m} at $A = q^{\frac{r+1}{2}}$ gives
 \[
  g_m = g_{m-1} \otimes f_1 - g'_{m-2}
 \]
 where
 \begin{equation}
  \fusion{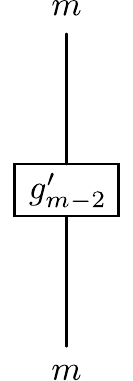} := -\frac{[m-1]}{[m]} \cdot \fusion{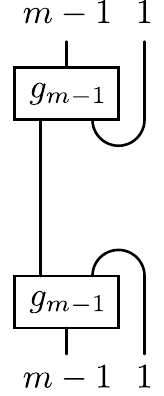} + \frac{2}{[m]^2} \cdot \fusion{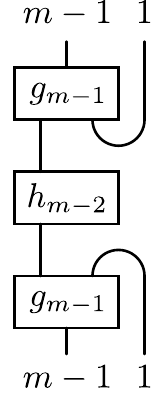} \label{E:g_prime_m-2}
 \end{equation}
 We know $F_{\TL}(g_{m-1} \otimes f_1)$ restricts to the identity on $P_{m-1} \otimes X_1$ and to zero on $Q_{m-1} \otimes X_1$ thanks to the induction hypothesis. Then, using equation~\eqref{E:P_m}, we have
 \begin{align*}
  \begin{split}
   b_0^m &= \frac{q^{-m}}{[m+1]} a_0^{m-1} \otimes a_1^1 + [m]q b_0^{m-1} \otimes a_1^1 + \frac{[m]}{[m+1]} b_1^{m-1} \otimes a_0^1 \\
   &= \left( \frac{[m] q^{-m+1}}{[m+1]} + q^{-m+2} \right) a_0^{m-2} \otimes a_1^1 \otimes a_1^1 + [m][m-1]q^2 b_0^{m-2} \otimes a_1^1 \otimes a_1^1 \\
   &\hspace*{\parindent} + \frac{q^{-m}}{[m+1]} a_1^{m-2} \otimes a_0^1 \otimes a_1^1 + [m-1]q b_1^{m-2} \otimes a_0^1 \otimes a_1^1 \\
   &\hspace*{\parindent} + \frac{q^{-m+1}}{[m+1]} a_1^{m-2} \otimes a_1^1 \otimes a_0^1 + [m-1]q^2 b_1^{m-2} \otimes a_1^1 \otimes a_0^1 \\
   &\hspace*{\parindent} + \frac{[m-1]}{[m+1]} b_2^{m-2} \otimes a_0^1 \otimes a_0^1,
  \end{split}
 \end{align*}
 This means that, using the definition of $F_\TL$ and equation~\eqref{E:g_prime_m-2}, we get
 \begin{align*}
  F_{\TL}(g'_{m-2})(b_0^m) &= 0, & F_{\TL}(g_m)(b_0^m) &= b_0^m.
 \end{align*}
 On the other hand, using equation~\eqref{E:tilde_P_m-2}, we have
 \begin{align*}
  \begin{split}
   {b'}_0^{m-2} &= \frac{q^m}{[m-1]} x_{m-r-1}^{m-1} \otimes a_1^1 - \frac{[m]}{[m-1]} b_0^{m-1} \otimes a_0^1 \\
   &= q^m x_{m-r-2}^{m-2} \otimes a_1^1 \otimes a_1^1 + \frac{q^m}{[m-1]} a_0^{m-2} \otimes a_0^1 \otimes a_1^1 \\
   &\hspace*{\parindent} - \frac{q^{-m+1}}{[m-1]} a_0^{m-2} \otimes a_1^1 \otimes a_0^1 - [m] q b_0^{m-2} \otimes a_1^1 \otimes a_0^1 \\
   &\hspace*{\parindent} - b_1^{m-2} \otimes a_0^1 \otimes a_0^1,
  \end{split}
 \end{align*}
 and, using equations~\eqref{E:P_m} and \eqref{E:tilde_P_m-2}, we have
 \begin{align*}
  a_0^{m-2} \otimes a_0^1 &= - {a'}_0^{m-3}, \\*
  b_0^{m-2} \otimes a_0^1 &= \frac{q^{m-1}}{[m-1]^2} x_{m-r-1}^{m-1} + \frac{q^{m-1}}{[m-1]^2} {a'}_0^{m-3} - \frac{[m-2]}{[m-1]} {b'}_0^{m-3}, \\*
  a_0^{m-2} \otimes a_1^1 &= \frac{1}{[m-1]} a_0^{m-1} + \frac{1}{[m-1]} {a'}_1^{m-3}, \\*
  b_0^{m-2} \otimes a_1^1 &= - \frac{\{ m-1 \}'}{[m-1]^3} a_0^{m-1} + \frac{[m]}{[m-1]^2} b_0^{m-1} - \frac{\{ m-1 \}'}{[m-1]^3} {a'}_1^{m-3} \\*
  &\hspace*{\parindent} + \frac{[m-2]}{[m-1]^2} {b'}_1^{m-3}.
 \end{align*}
 This means that, using the definition of $F_\TL$ and equation~\eqref{E:g_prime_m-2}, we get
 \begin{align*}
  F_{\TL}(g'_{m-2})({b'}_0^{m-2}) &= {b'}_0^{m-2}, &
  F_{\TL}(g_m)({b'}_0^{m-2}) &= 0.
 \end{align*}
 Therefore, $F_{\TL}(g_m) = \psi_m$. This proves equation~\eqref{E:F(g_m)}. For what concerns equation~\eqref{E:F(p_m)}, using equation~\eqref{E:P_m}, we have
 \[
  b_0^m = \frac{q^{-m}}{[m+1]} a_0^{m-1} \otimes a_1^1 + [m]q b_0^{m-1} \otimes a_1^1 + \frac{[m]}{[m+1]} b_1^{m-1} \otimes a_0^1,
 \]
 and, using equation~\eqref{E:X_m}, we have
 \begin{align*}
  a_0^{2r-m-1} &= a_0^{2r-m-2} \otimes a_0^1, &
  a_1^{2r-m-1} &= q^{m+2} a_0^{2r-m-2} \otimes a_1^1 + a_1^{2r-m-2} \otimes a_0^1.
 \end{align*}
 Then, using the definition of $F_\TL$ and equation~\eqref{E:p_i_m}, this means
 \[
  F_{\TL}(p_m)(b_0^m) = (-1)^m \frac{[m-r]!}{[m+1]} a_0^{2r-m-2}.
 \]
 For what concerns equation~\eqref{E:F(i_m)}, using equations~\eqref{E:X_m} and \eqref{E:tilde_X_m-2}, we have
 \begin{align*}
  a_0^{2r-m-2} \otimes a_0^1 &= a_0^{2r-m-1}, \\*
  a_0^{2r-m-2} \otimes a_1^1 &= -\frac{1}{[m+1]} a_1^{2r-m-1} - \frac{1}{[m+1]} {a'}_0^{2r-m-3},
 \end{align*}
 and, using equation~\eqref{E:P_m}, we have
 \[
  a_0^m = [m+1] q a_0^{m-1} \otimes a_1^1 + a_1^{m-1} \otimes a_0^1.
 \]
 Then, using the definition of $F_\TL$ and equation~\eqref{E:p_i_m}, this means
 \[
  F_{\TL}(i_m)(a_0^{2r-m-2}) = \frac{1}{[m-r+1]!} a_0^m.
 \]
 Now equation~\eqref{E:F(h_m)} is a consequence of equations~\eqref{E:F(p_m)} and \eqref{E:F(i_m)}. Next, let us denote with $Y_{2r-1}$ the unique submodule of $X^{\otimes 2r-1}$ satisfying $X^{\otimes 2r-1} \cong X_{2r-1} \oplus Y_{2r-1}$. For $m = 2r-1$, specializing equation~\eqref{E:f_2r-1} at $A = q^{\frac{r+1}{2}}$ gives
 \[
  f_{2r-1} = g_{2r-2} \otimes f_1 - g'_{2r-3}
 \]
 where
 \begin{equation}
  \fusion{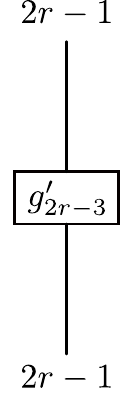} := -[2] \cdot \fusion{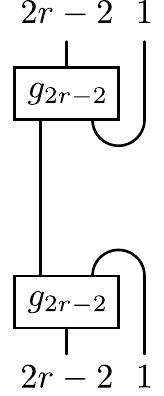} + 2 \cdot \fusion{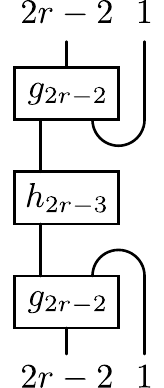} \label{E:g_prime_2r-3}
 \end{equation}
 We know $F_{\TL}(g_{2r-2} \otimes f_1)$ restricts to the identity on $P_{2r-2} \otimes X_1$ and to zero on $Q_{2r-2} \otimes X_1$ thanks to equation~\eqref{E:F(g_m)}. Then, using equation~\eqref{E:X_r-1_+_-}, we have
 \begin{align*}
  \begin{split}
   a_0^{r-1,+} &= x_0^{2r-2} \otimes a_0^1 = x_0^{2r-3} \otimes a_0^1 \otimes a_0^1, \\
   a_0^{r-1,-} &= q a_0^{2r-2} \otimes a_1^1 - y_0^{2r-2} \otimes a_0^1 \\
   &= - q^2 a_0^{2r-3} \otimes a_1^1 \otimes a_1^1 + q a_1^{2r-3} \otimes a_0^1 \otimes a_1^1 \\
   &\hspace*{\parindent} + q^2 a_1^{2r-3} \otimes a_1^1 \otimes a_0^1 - [2] y_0^{2r-3} \otimes a_0^1 \otimes a_0^1
  \end{split}
 \end{align*}
 This means that, using the definition of $F_\TL$ and equation~\eqref{E:g_prime_2r-3}, we get
 \begin{align*}
  F_{\TL}(g'_{2r-3})(a_0^{r-1,+}) &= 0, & F_{\TL}(f_{2r-1})(a_0^{r-1,+}) = a_0^{r-1,+}, \\*
  F_{\TL}(g'_{2r-3})(a_0^{r-1,-}) &= 0, & F_{\TL}(f_{2r-1})(a_0^{r-1,-}) = a_0^{r-1,-}.
 \end{align*}
 On the other hand, using equation~\eqref{E:tilde_P_m-2}, we have
 \begin{align*}
  \begin{split}
   {b'}_0^{2r-3} &= -\frac{q^{-1}}{[2]} x_{r-2}^{2r-2} \otimes a_1^1 - \frac{1}{[2]} b_0^{2r-2} \otimes a_0^1 \\
   &= q^{-1} x_{r-3}^{2r-3} \otimes a_1^1 \otimes a_1^1 - \frac{q^{-1}}{[2]} a_0^{2r-3} \otimes a_0^1 \otimes a_1^1 \\
   &\hspace*{\parindent} + \frac{q^2}{[2]} a_0^{2r-3} \otimes a_1^1 \otimes a_0^1 + q b_0^{2r-3} \otimes a_1^1 \otimes a_0^1 \\
   &\hspace*{\parindent} - b_1^{2r-3} \otimes a_0^1 \otimes a_0^1,
  \end{split}
 \end{align*}
 and, using equations~\eqref{E:P_m} and \eqref{E:tilde_P_m-2}, we have
 \begin{align*}
  a_0^{2r-3} \otimes a_0^1 &= - {a'}_0^{2r-4}, \\*
  b_0^{2r-3} \otimes a_0^1 &= \frac{q^{-2}}{[2]^2} x_{r-2}^{2r-2} + \frac{q^{-2}}{[2]^2} {a'}_0^{2r-4} - \frac{[3]}{[2]} {b'}_0^{2r-4}, \\*
  a_0^{2r-3} \otimes a_1^1 &= -\frac{1}{[2]} a_0^{2r-2} - \frac{1}{[2]} {a'}_1^{2r-4}, \\*
  b_0^{2r-3} \otimes a_1^1 &= \frac{\{ 2 \}'}{[2]^3} a_0^{2r-2} - \frac{1}{[2]^2} b_0^{2r-2} + \frac{\{ 2 \}'}{[2]^3} {a'}_1^{2r-4} - \frac{[3]}{[2]^2} {b'}_1^{2r-4}.
 \end{align*}
 This means that, using the definition of $F_\TL$ and equation~\eqref{E:g_prime_2r-3}, we get
 \begin{align*}
  F_{\TL}(g'_{2r-3})({b'}_0^{2r-3}) &= {b'}_0^{2r-3}, &
  F_{\TL}(f_{2r-1})({b'}_0^{2r-3}) &= 0.
 \end{align*}
 Therefore, $F_{\TL}(f_{2r-1}) = \varphi_{2r-1}^+ + \varphi_{2r-1}^-$. This proves equation~\eqref{E:F(f_2r-1)}.
\end{proof}

The functor $F_{\TL} : \TL \to \mods{\bar{U}}$ is not full. Indeed, we have 
\[
 \dim_\C \TL(g_m,g_{3r-m-2}) = 0
\] 
for every integer $r \leqs m \leqs 2r-2$, as follows directly from a parity argument, $3r-2m-2$ being odd. This does not match the dimension of $\Hom_{\bar{U}}(P_m,P_{3r-m-2})$, as recalled in Section~\ref{S:small_quantum_group}.

\section{Extended Temperley--Lieb category}\label{S:extended_TL_category}

In this section we introduce an extended version of the category $\TL$ over which we define a monoidal linear functor with target $\mods{\bar{U}}$.

\begin{definition}\label{D:extended_TL}
 The \textit{extended Temperley--Lieb category} $\tTL$ is the smallest linear category containing $\TL$ as a linear subcategory, as well as morphisms 
 \[
  p_{2r-1}^+, p_{2r-1}^- \in \tTL(f_{2r-1},f_{r-1}), \quad i_{2r-1}^+, i_{2r-1}^- \in \tTL(f_{r-1},f_{2r-1})
 \]
 satisfying, for every integer $0 \leqs m \leqs r-1$ and all $\varepsilon, \varepsilon' \in \{ +,- \}$, equations
 \begin{align}
  \etl{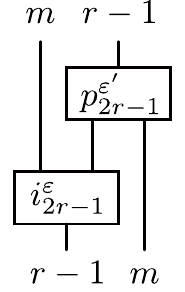} \hspace*{-5pt} &= \delta_{\varepsilon,\varepsilon'} \cdot \hspace*{-5pt} \etl{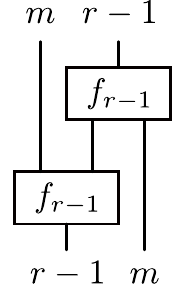} &
  \etl{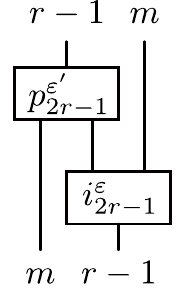} \hspace*{-5pt} &= \delta_{\varepsilon,\varepsilon'} \cdot \hspace*{-5pt} \etl{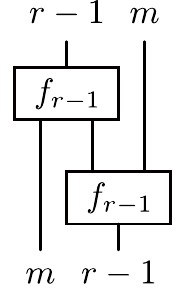} \label{E:p_i} \\
  \hspace*{-5pt} \sum_{\varepsilon \in \{ +,- \}} \etl{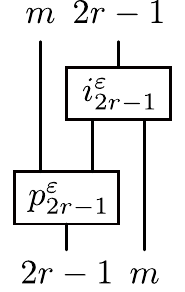} \hspace*{-5pt} &= \hspace*{-5pt} \etl{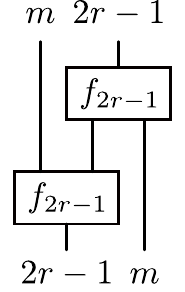} &
  \hspace*{-10pt} \sum_{\varepsilon \in \{ +,- \}} \hspace*{-5pt} \etl{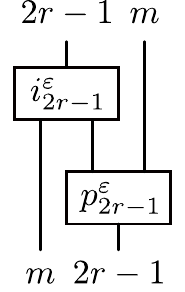} \hspace*{-5pt} &= \hspace*{-5pt} \etl{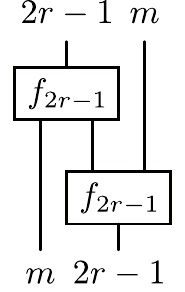} \label{E:i_p} \\
  \etl{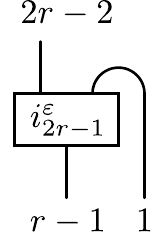} &= - \etl{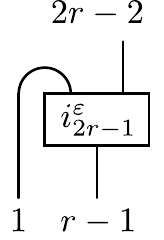} &  
  \etl{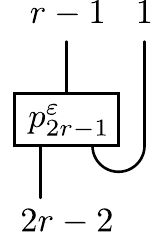} &= - \etl{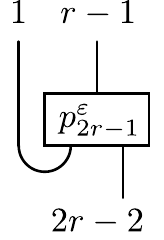} \label{E:p_i_ambi}
 \end{align}
\end{definition}

\begin{remark}\label{R:p_i_ptr}
 As an immediate consequence of equation~\eqref{E:p_i_ambi}, for every $\varepsilon \in \{ +,- \}$ we have
 \begin{align}
  \etl{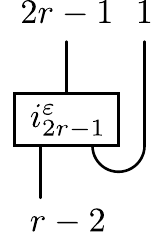} &= - \etl{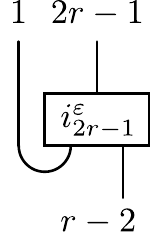} &  
  \etl{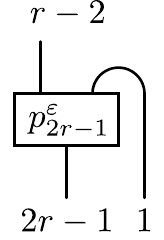} &= - \etl{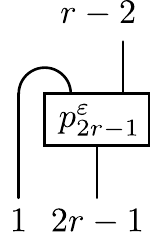} \label{E:p_i_ambi_prime} \\
  \etl{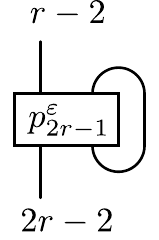} &= 0 & \etl{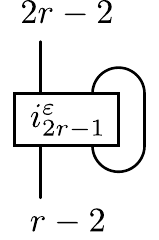} &= 0 \label{E:p_i_ptr}
 \end{align}
\end{remark}

We define a monoidal linear functor
\[
 F_{\tTL} : \tTL \to \mods{\bar{U}}
\]
extending $F_{\TL} : \TL \to \mods{\bar{U}}$ by setting
\[
 F_{\tTL}(p_{2r-1}^\varepsilon) := \pi_{2r-1}^\varepsilon, \quad
 F_{\tTL}(i_{2r-1}^\varepsilon) := \iota_{2r-1}^\varepsilon
\]
for all $\varepsilon, \varepsilon' \in \{ +,- \}$.

\begin{lemma}\label{L:well-def}
 The monoidal linear functor $F_{\tTL} : \tTL \to \mods{\bar{U}}$ is well-defined.
\end{lemma}

The proof of Lemma~\ref{L:well-def} is a lengthy computation which will occupy the remainder of this section.

\begin{proof}
 In order to discuss the proof, let us first choose names for the morphisms appearing in equations~\eqref{E:p_i} and \eqref{E:i_p}, which we rewrite as
 \begin{align*}
  t_{r-1,m}^{\varepsilon,\varepsilon'} &= \delta_{\varepsilon,\varepsilon'} \cdot t_{r-1,m}, &
  u_{r-1,m}^{\varepsilon,\varepsilon'} &= \delta_{\varepsilon,\varepsilon'} \cdot u_{r-1,m}, \\*
  \sum_{\varepsilon \in \{ +,- \}} t_{2r-1,m}^\varepsilon &= t_{2r-1,m}, &
  \sum_{\varepsilon \in \{ +,- \}} u_{2r-1,m}^\varepsilon &= u_{2r-1,m},
 \end{align*}
 respectively. We also give names to the images of these morphisms under the functor $F_{\tTL}$ by setting
 \begin{align*}
  \tau_{r-1,m}^{\varepsilon,\smash{\varepsilon'}} &:= F_{\tTL}(t_{r-1,m}^{\varepsilon,\smash{\varepsilon'}}), &
  \tau_{r-1,m} &:= F_{\tTL}(t_{r-1,m}), \\*
  \upsilon_{r-1,m}^{\varepsilon,\smash{\varepsilon'}} &:= F_{\tTL}(u_{r-1,m}^{\varepsilon,\smash{\varepsilon'}}), &
  \upsilon_{r-1,m} &:= F_{\tTL}(u_{r-1,m}), \\*
  \tau_{2r-1,m}^\varepsilon &:= F_{\tTL}(t_{2r-1,m}^\varepsilon), &
  \tau_{2r-1,m} &:= F_{\tTL}(t_{2r-1,m}), \\*
  \upsilon_{2r-1,m}^\varepsilon &:= F_{\tTL}(u_{2r-1,m}^\varepsilon), &
  \upsilon_{2r-1,m} &:= F_{\tTL}(u_{2r-1,m}).
 \end{align*}
 Therefore, what we need to show is
 \begin{align*}
  \tau_{r-1,m}^{\varepsilon,\varepsilon'} &= \delta_{\varepsilon,\varepsilon'} \tau_{r-1,m}, &
  \upsilon_{r-1,m}^{\varepsilon,\varepsilon'} &= \delta_{\varepsilon,\varepsilon'} \upsilon_{r-1,m}, \\
  \sum_{\varepsilon \in \{ +,- \}} \tau_{2r-1,m}^\varepsilon &= \tau_{2r-1,m}, &
  \sum_{\varepsilon \in \{ +,- \}} \upsilon_{2r-1,m}^\varepsilon &= \upsilon_{2r-1,m}.
 \end{align*} 
 When $m = 0$, the relations reduce to equations~\eqref{E:pi_iota_rel_1} and \eqref{E:pi_iota_rel_2}. Then, let us suppose $0 <m \leqs r-1$. Thanks to equation~\eqref{E:a_j^r-1_+_-}, we have
 \begin{align*}
  \begin{split}
   a_0^{r-1} &= a_0^m \otimes a_0^{r-m-1} = a_0^{r-m-1} \otimes a_0^m, \\
   a_0^{r-1,+} &= a_0^m \otimes x_0^{2r-m-1} = x_0^{2r-m-1} \otimes a_0^m, \\
   a_{r-1}^{r-1} &= (-1)^m a_m^m \otimes a_{r-m-1}^{r-m-1} = (-1)^m a_{r-m-1}^{r-m-1} \otimes a_m^m, \\
   a_{r-1}^{r-1,-} &= (-1)^{m+1} [m] a_m^m \otimes y_{r-m-1}^{2r-m-1} = (-1)^{m+1} [m] y_{r-m-1}^{2r-m-1} \otimes a_m^m.
  \end{split}
 \end{align*}
 Then, for all integers $0 \leqs j \leqs m$ and $0 \leqs k \leqs r-1$, thanks to equations~\eqref{E:alpha_r-1} to \eqref{E:alpha_2r-1_-} and \eqref{E:a_j^r-1_+_-}, we have
 \begin{align*}
  &\tau_{r-1,m}^{+,+}(a_k^{r-1} \otimes a_j^m) \\*
  &\hspace*{\parindent} = \sum_{\ell = \max \{ m+k-r+1,0 \}}^{\min \{ m,k \}} \frac{[k]!}{[\ell]![k-\ell]!} q^{-(k-\ell)(m-\ell)} a_\ell^m \otimes \pi_{2r-1}^+(x_{k-\ell}^{2r-m-1} \otimes a_j^m) \\*
  &\hspace*{\parindent} = \sum_{\ell = \max \{ m+k-r+1,0 \}}^{\min \{ m,k \}} \frac{[k]!}{[\ell]![k-\ell]!} q^{-(k-\ell)(m-\ell)} a_\ell^m \otimes \varphi_{r-1}(a_{k-\ell}^{r-m-1} \otimes a_j^m) \\*
  &\hspace*{\parindent} = \tau_{r-1,m}(a_k^{r-1} \otimes a_j^m),
 \end{align*}
 we have
 \begin{align*}
  &\tau_{r-1,m}^{-,-}(a_k^{r-1} \otimes a_j^m) \\*
  &\hspace*{\parindent} = -[m] \sum_{\ell = \max \{ m+k-r+1,0 \}}^{\min \{ m,k \}} \frac{[k]!}{[\ell]![k-\ell]!} q^{-(k-\ell)(m-\ell)} a_\ell^m \otimes \pi_{2r-1}^-(y_{k-\ell}^{2r-m-1} \otimes a_j^m) \\*
  &\hspace*{\parindent} = \sum_{\ell = \max \{ m+k-r+1,0 \}}^{\min \{ m,k \}} \frac{[k]!}{[\ell]![k-\ell]!} q^{-(k-\ell)(m-\ell)} a_\ell^m \otimes \varphi_{r-1}(a_{k-\ell}^{r-m-1} \otimes a_j^m) \\*
  &\hspace*{\parindent} = \tau_{r-1,m}(a_k^{r-1} \otimes a_j^m),
 \end{align*}
 we have
 \begin{align*}
  &\upsilon_{r-1,m}^{+,+}(a_j^m \otimes a_k^{r-1}) \\*
  &\hspace*{\parindent} = \sum_{\ell = \max \{ m+k-r+1,0 \}}^{\min \{ m,k \}} \frac{[k]!}{[\ell]![k-\ell]!} q^{(m+k-\ell+1)\ell} \pi_{2r-1}^+(a_j^m \otimes x_{k-\ell}^{2r-m-1}) \otimes a_\ell^m \\*
  &\hspace*{\parindent} = \sum_{\ell = \max \{ m+k-r+1,0 \}}^{\min \{ m,k \}} \frac{[k]!}{[\ell]![k-\ell]!} q^{(m+k-\ell+1)\ell} \varphi_{r-1}(a_j^m \otimes a_{k-\ell}^{r-m-1}) \otimes a_\ell^m \\*
  &\hspace*{\parindent} = \upsilon_{r-1,m}(a_j^m \otimes a_k^{r-1}),
 \end{align*}
 we have
 \begin{align*}
  &\upsilon_{r-1,m}^{-,-}(a_j^m \otimes a_k^{r-1}) \\*
  &\hspace*{\parindent} = -[m] \sum_{\ell = \max \{ m+k-r+1,0 \}}^{\min \{ m,k \}} \frac{[k]!}{[\ell]![k-\ell]!} q^{(m+k-\ell+1)\ell} \pi_{2r-1}^-(a_j^m \otimes y_{k-\ell}^{2r-m-1}) \otimes a_\ell^m \\*
  &\hspace*{\parindent} = \sum_{\ell = \max \{ m+k-r+1,0 \}}^{\min \{ m,k \}} \frac{[k]!}{[\ell]![k-\ell]!} q^{(m+k-\ell+1)\ell} \varphi_{r-1}(a_j^m \otimes a_{k-\ell}^{r-m-1}) \otimes a_\ell^m \\*
  &\hspace*{\parindent} = \upsilon_{r-1,m}(a_j^m \otimes a_k^{r-1}),
 \end{align*}
 we have
 \begin{align*}
  \tau_{r-1,m}^{+,-}(a_k^{r-1} \otimes a_j^m) &= \tau_{r-1,m}^{-,+}(a_k^{r-1} \otimes a_j^m) = 0, \\*
  \upsilon_{r-1,m}^{+,-}(a_j^m \otimes a_k^{r-1}) &= \upsilon_{r-1,m}^{-,+}(a_j^m \otimes a_k^{r-1}) = 0,
 \end{align*}
 we have
 \begin{align*}
  &\tau_{2r-1,m}^+(a_k^{r-1,+} \otimes a_j^m) \\*
  &\hspace*{\parindent} = \sum_{\ell = \max \{ m+k-r+1,0 \}}^{\min \{ m,k \}} \frac{[k]!}{[\ell]![k-\ell]!} q^{-(k-\ell)(m-\ell)} a_\ell^m \otimes \iota_{2r-1}^+(a_{k-\ell}^{r-m-1} \otimes a_j^m) \\*
  &\hspace*{\parindent} = \sum_{\ell = \max \{ m+k-r+1,0 \}}^{\min \{ m,k \}} \frac{[k]!}{[\ell]![k-\ell]!} q^{-(k-\ell)(m-\ell)} a_\ell^m \otimes \varphi_{2r-1}^+(x_{k-\ell}^{2r-m-1} \otimes a_j^m) \\*
  &\hspace*{\parindent} = \tau_{2r-1,m}(a_k^{r-1+} \otimes a_j^m),
 \end{align*}
 we have
 \begin{align*}
  &\tau_{2r-1,m}^-(a_k^{r-1,-} \otimes a_j^m) \\*
  &\hspace*{\parindent} = \sum_{\ell = \max \{ m+k-r+1,0 \}}^{\min \{ m,k \}} \frac{[k]!}{[\ell]![k-\ell]!} q^{-(k-\ell)(m-\ell)} a_\ell^m \otimes \iota_{2r-1}^-(a_{k-\ell}^{r-m-1} \otimes a_j^m) \\*
  &\hspace*{\parindent} = -[m] \sum_{\ell = \max \{ m+k-r+1,0 \}}^{\min \{ m,k \}} \frac{[k]!}{[\ell]![k-\ell]!} q^{-(k-\ell)(m-\ell)} a_\ell^m \otimes \varphi_{2r-1}^-(y_{k-\ell}^{2r-m-1} \otimes a_j^m) \\*
  &\hspace*{\parindent} = \tau_{2r-1,m}(a_k^{r-1,-} \otimes a_j^m),
 \end{align*}
 we have
 \begin{align*}
  &\upsilon_{r-1,m}^+(a_j^m \otimes a_k^{r-1,+}) \\*
  &\hspace*{\parindent} = \sum_{\ell = \max \{ m+k-r+1,0 \}}^{\min \{ m,k \}} \frac{[k]!}{[\ell]![k-\ell]!} q^{(m+k-\ell+1)\ell} \iota_{2r-1}^+(a_j^m \otimes a_{k-\ell}^{r-m-1}) \otimes a_\ell^m \\*
  &\hspace*{\parindent} = \sum_{\ell = \max \{ m+k-r+1,0 \}}^{\min \{ m,k \}} \frac{[k]!}{[\ell]![k-\ell]!} q^{(m+k-\ell+1)\ell} \varphi_{2r-1}^+(a_j^m \otimes x_{k-\ell}^{2r-m-1}) \otimes a_\ell^m \\*
  &\hspace*{\parindent} = \upsilon_{2r-1,m}(a_j^m \otimes a_k^{r-1,+}),
 \end{align*}
 we have
 \begin{align*}
  &\upsilon_{2r-1,m}^-(a_j^m \otimes a_k^{r-1-}) \\*
  &\hspace*{\parindent} = \sum_{\ell = \max \{ m+k-r+1,0 \}}^{\min \{ m,k \}} \frac{[k]!}{[\ell]![k-\ell]!} q^{(m+k-\ell+1)\ell} \iota_{2r-1}^-(a_j^m \otimes a_{k-\ell}^{r-m-1}) \otimes a_\ell^m \\*
  &\hspace*{\parindent} = -[m] \sum_{\ell = \max \{ m+k-r+1,0 \}}^{\min \{ m,k \}} \frac{[k]!}{[\ell]![k-\ell]!} q^{(m+k-\ell+1)\ell} \varphi_{2r-1}^-(a_j^m \otimes y_{k-\ell}^{2r-m-1}) \otimes a_\ell^m \\*
  &\hspace*{\parindent} = \upsilon_{2r-1,m}(a_j^m \otimes a_k^{r-1,-}),
 \end{align*}
 and we have
 \begin{align*}
  \tau_{2r-1,m}^+(a_k^{r-1,-} \otimes a_j^m) &= \tau_{2r-1,m}^-(a_k^{r-1,+} \otimes a_j^m) = 0, \\*
  \upsilon_{r-1,m}^+(a_j^m \otimes a_k^{r-1,-}) &= \upsilon_{r-1,m}^-(a_j^m \otimes a_k^{r-1,+}) = 0.
 \end{align*}
 Similarly, let us choose names for the morphisms appearing in equation~\eqref{E:p_i_ambi}, which we rewrite as
 \[
  c_r^\varepsilon = - \tilde{c}_r^\varepsilon, \qquad \qquad
  d_r^\varepsilon = - \tilde{d}_r^\varepsilon.
 \]
 Equations~\eqref{E:X_m}, \eqref{E:P_r}, and \eqref{E:a_j^r-1_+_-} give
 \begin{align*}
  b_0^r &= a_0^{r-1} \otimes a_1^1 = a_0^1 \otimes a_0^{r-1} \otimes a_1^1, \\
  a_0^{r-1,+} &= x_0^{2r-2} \otimes a_0^1, \\
  a_0^{r-1,-} &= q a_0^{2r-2} \otimes a_1^1 - y_0^{2r-2} \otimes a_0^1, \\
  a_1^{r-1,+} &= q^{-1} a_0^1 \otimes x_1^{2r-2} + a_1^1 \otimes x_0^{2r-2}, \\
  a_1^{r-1,-} &= -q^{-1} a_0^1 \otimes y_1^{2r-2} - a_1^1 \otimes y_0^{2r-2},
 \end{align*} 
 and so equation~\eqref{E:alpha_r-1} gives
 \begin{align*}
  F_{\tTL}(c_r^+)(b_0^r) &= - x_0^{2r-2}, &
  F_{\tTL}(\tilde{c}_r^+)(b_0^r) &= x_0^{2r-2}, \\
  F_{\tTL}(c_r^-)(b_0^r) &= y_0^{2r-2}, &
  F_{\tTL}(\tilde{c}_r^-)(b_0^r) &= - y_0^{2r-2}.
 \end{align*}
 Furthermore, equations~\eqref{E:v_0_m_simple} and \eqref{E:v_0_m_projective} give
 \[
  x_0^r = a_0^{r-1} \otimes a_0^1 = a_0^1 \otimes a_0^{r-1}, \qquad
  y_0^r = a_{r-1}^{r-1} \otimes a_1^1 = a_1^1 \otimes a_{r-1}^{r-1},
 \]
 and so equations~\eqref{E:alpha_2r-1_+} and \eqref{E:alpha_2r-1_-} give
 \begin{align*}
  F_{\tTL}(d_r^+)(b_0^{2r-2}) &= y_0^r, &
  F_{\tTL}(\tilde{d}_r^+)(b_0^{2r-2}) &= - y_0^r, \\
  F_{\tTL}(d_r^-)(b_0^{2r-2}) &= - x_0^r, &
  F_{\tTL}(\tilde{d}_r^-)(b_0^{2r-2}) &= x_0^r. \qedhere
 \end{align*}
\end{proof}

\section{Proof of Theorem~\ref{T:functor}}\label{S:proof}

In this section we establish our main result. We do this in two main steps. First of all, for convenience of notation, let us set 
\[
 t_m :=
 \begin{cases}
  f_m & 0 \leqs m \leqs r-1 \\
  g_m & r \leqs m \leqs 2r-2
 \end{cases} \in \tTL
\]
for all integers $0 \leqs m \leqs 2r-2$. Then, let us show the family of objects 
\[
 \{ t_m \in \tTL \mid 0 \leqs m \leqs 2r-2 \}
\]
dominates $\tTL$.

\begin{proposition}\label{P:domination}
 For every $m \in \tTL$ there exists a finite set $J_{m,n}$ for every integer $0 \leqs n \leqs 2r-2$ and morphisms $u_m^{n,j} \in \tTL(m,n)$ and $v^m_{n,j} \in \tTL(n,m)$ for every $j \in J_{m,n}$ satisfying
 \[
  \id_m = \sum_{n=0}^{2r-2} \sum_{j \in J_{m,n}} v^m_{n,j} t_n u_m^{n,j}.
 \]
\end{proposition}

\begin{proof}
 The statement is proved by induction on $m$. For $m = 0$, we simply need to set $\left| J_{0,n} \right| = \delta_{0,n}$ and $u_0^0 = v_0^0 = f_0$. For $m > 0$, we decompose $\id_m$ as
 \[
  \domination{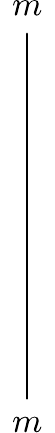} = \sum_{n=0}^{2r-2} \sum_{j \in J_{m-1,n}} \domination{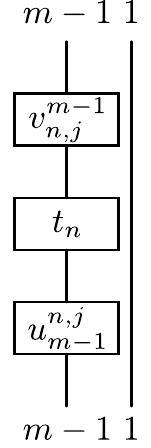}
 \]
 Equations~\eqref{E:f_def} and \eqref{E:g_r} allow us to rewrite $f_n \otimes f_1$ for every integer $0 \leqs n \leqs r-1$, equations~\eqref{E:g_r+1} and \eqref{E:g_m} allow us to rewrite $g_n \otimes f_1$ for every integer $r \leqs n \leqs 2r-3$, and furthermore, since equation~\eqref{E:i_p} for $m=0$ implies in particular
 \[
  \domination{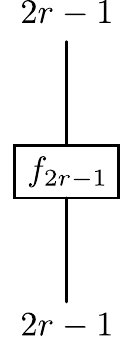} = \domination{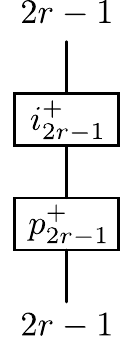} + \domination{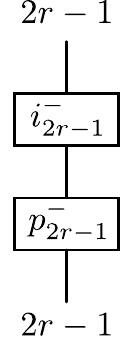}
 \]
 equation~\eqref{E:f_2r-1} allows us to rewrite $g_{2r-2} \otimes f_1$, something we would not be able to do in $\TL$. For every integer $0 \leqs n < 2r-2$ and every $j \in J_{m-1,n}$, we set
 \[
  \domination{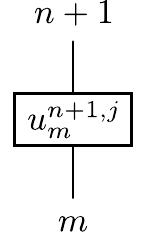} := \domination{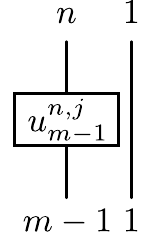} \qquad
  \domination{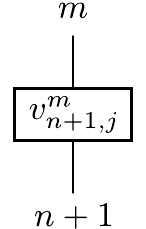} := \domination{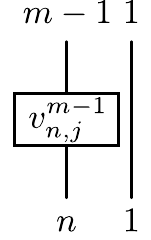}
 \]
 On the other hand, for $n = 2r-2$ and for every $j \in J_{m-1,2r-2}$, we set
 \begin{gather*}
  \domination{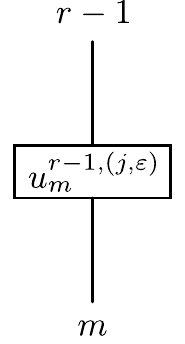} := \domination{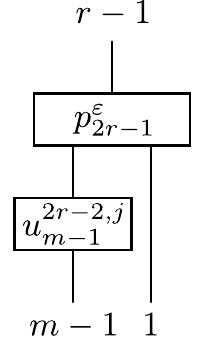} \qquad
  \domination{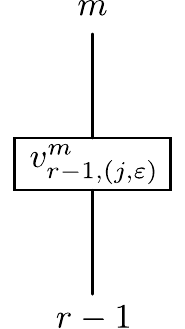} := \domination{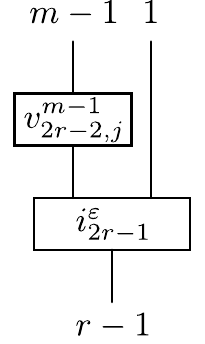} 
 \end{gather*}
 When $0 < n < r-1$, we set
 \[
  \domination{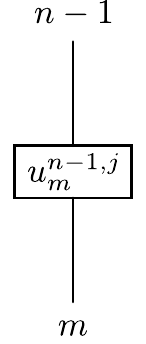} := -\frac{[n]}{[n+1]} \cdot \domination{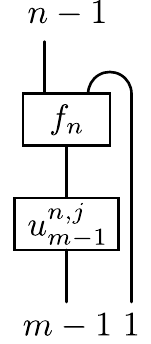} \qquad
  \domination{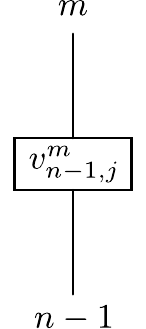} := \domination{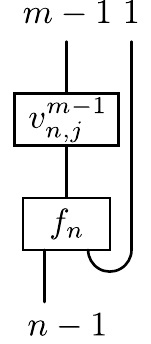}
 \]
 for every $j \in J_{m-1,n}$. When $n = r$, we set
 \begin{gather*}
  \domination{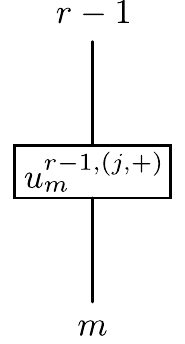} := \domination{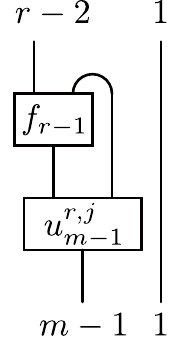} \qquad
  \domination{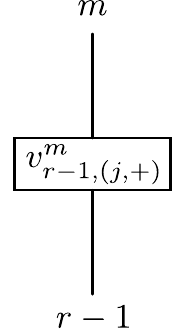} := \domination{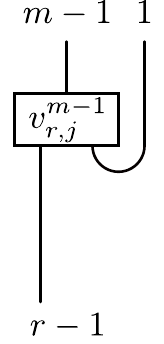} \\
  \domination{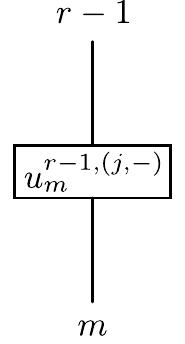} := \domination{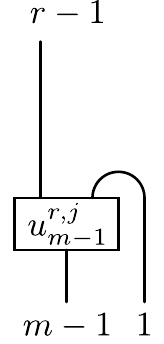} + [2] \cdot \domination{u_m_r-1_j,+_a} \qquad
  \domination{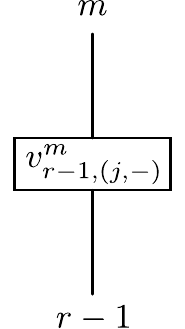} := \domination{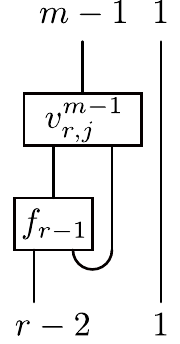}
 \end{gather*}
 for every $j \in J_{m-1,r}$. When $r < n \leqs 2r-2$, we set
 \[
  \domination{u_m_n-1_j} := -\frac{[n]}{[n+1]} \cdot \domination{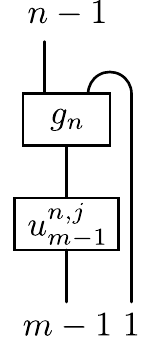} + \frac{2}{[n+1]^2} \cdot \domination{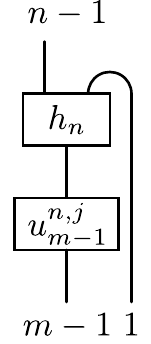} \qquad
  \domination{v_m_n-1_j} := \domination{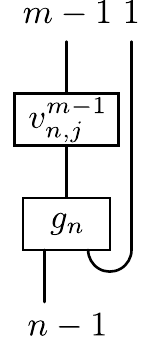}
 \]
 for every $j \in J_{m-1,n}$. Therefore we can set
 \[
  J_{m,n} := 
  \begin{cases}
   J_{m-1,1} & n = 0, \\
   J_{m-1,n-1} \sqcup J_{m-1,n+1} & 0 < n < r-1, \\
   J_{m-1,r-2} \sqcup \left( ( J_{m-1,r} \sqcup J_{m-1,2r-2}) \times \{ +,- \} \right) & n = r-1, \\
   J_{m-1,r-1} & n = r, \\
   J_{m-1,n-1} \sqcup J_{m-1,n+1} & r < n < 2r-2, \\
   J_{m-1,2r-3} & n = 2r-2.
  \end{cases} \qedhere
 \]
\end{proof}

The second step in the proof of Theorem~\ref{T:functor} consists in exhibiting morphisms of $\tTL$ whose image under $F_{\tTL}$ recovers the missing morphisms between indecomposable projective $\bar{U}$-modules. In order to do this, let us define morphisms $c_m^\varepsilon \in \tTL(g_m,g_{3r-m-2})$ and $d_m^\varepsilon \in \tTL(g_{3r-m-2},g_m)$ as
\begin{equation}
 \etl{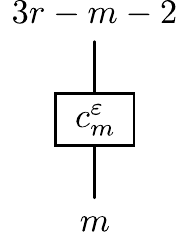} \hspace*{-7.5pt} := \etl{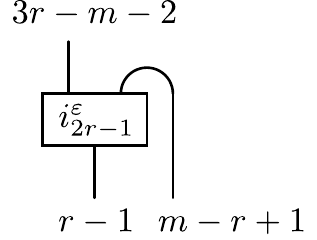} \qquad
 \etl{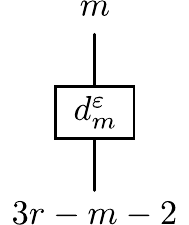} \hspace*{-7.5pt} := \etl{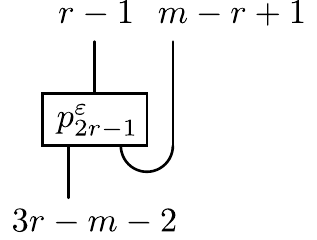} \hspace*{-7.5pt} \label{E:c_d_def} 
\end{equation}
for every integer $r \leqs m \leqs 2r-2$, which immediately gives
\begin{equation}
 \etl{c_m} \hspace*{-7.5pt} = \etl{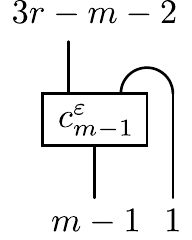} \qquad
 \etl{d_m} \hspace*{-7.5pt} = \etl{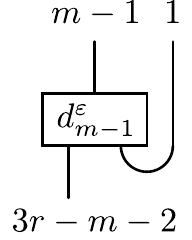} \label{E:c_d_m}
\end{equation}
The fact that $c_m^\varepsilon g_m = c_m^\varepsilon$ and that $g_m d_m^\varepsilon = d_m^\varepsilon$ can be proved by induction using equations~\eqref{E:g_r} to \eqref{E:g_m}, \eqref{E:p_i_ptr}, and \eqref{E:c_d_m}

\begin{lemma}\label{L:tilde_functor}
  For every integer $r \leqs m \leqs 2r-2$ and every $\varepsilon \in \{ +,- \}$ the morphisms $c_m^\varepsilon$ and $d_m^\varepsilon$ satisfy
 \begin{align}
  F_{\tTL}(c_m^+) &= (-1)^m \frac{[m-r]!}{[m+1]} \gamma_m^+, \label{E:F(c+)} \\*
  F_{\tTL}(c_m^-) &= -(-1)^m [m-r]! \gamma_m^-, \label{E:F(c-)} \\*
  F_{\tTL}(d_m^+) &= \frac{1}{[m-r+1]!} \gamma_{3r-m-2}^-, \label{E:F(d+)} \\*
  F_{\tTL}(d_m^-) &= -\frac{1}{[m-r+1]![m+1]} \gamma_{3r-m-2}^+, \label{E:F(d-)}
 \end{align}
 where $\gamma_m^+$ and $\gamma_m^-$ are defined by equations~\eqref{E:gamma_m_+_def} and \eqref{E:gamma_m_-_def}.
\end{lemma}

\begin{proof}
 The statement is proved by induction on $m$. For $m = r$ it follows from computations in the proof of Lemma~\ref{L:well-def} establishing the image under $F_{\tTL}$ of equation~\eqref{E:p_i_ambi}. For $r < m \leqs 2r-2$, equation~\eqref{E:P_m} gives
 \begin{align*}
  b_0^m &= \frac{q^{-m}}{[m+1]} a_0^{m-1} \otimes a_1^1 + [m]q b_0^{m-1} \otimes a_1^1 + \frac{[m]}{[m+1]} b_1^{m-1} \otimes a_0^1, \\*
  x_0^{3r-m-1} &= x_0^{3r-m-2} \otimes a_0^1, \\*
  x_1^{3r-m-1} &= q^{m+2} x_0^{3r-m-2} \otimes a_1^1 + x_1^{3r-m-2} \otimes a_0^1, \\*
  y_0^{3r-m-1} &= -\frac{q^{m+1}}{[m]} a_{m-r}^{3r-m-2} \otimes a_1^1 + \frac{[m+1]}{[m]} y_0^{3r-m-2} \otimes a_0^1, \\*
  y_1^{3r-m-1} &= \frac{[m+1]q^{m+2}}{[m]} y_0^{3r-m-2} \otimes a_1^1 + \frac{[m+1]}{[m]} y_1^{3r-m-2} \otimes a_0^1,
 \end{align*}
 and equation~\eqref{E:c_d_m} gives
 \begin{align*}
  F_{\tTL}(c_m^+)(b_0^m) &= (-1)^m \frac{[m-r]!}{[m+1]} x_0^{3r-m-2}, \\*
  F_{\tTL}(c_m^-)(b_0^m) &= -(-1)^m [m-r]! y_0^{3r-m-2}.
 \end{align*}
 This proves equations~\eqref{E:F(c+)} and \eqref{E:F(c-)}. Furthermore, equations~\eqref{E:P_m} and \eqref{E:tilde_P_m-2} give
 \begin{align*}
  b_0^{3r-m-2} \otimes a_0^1 &= \frac{q^{-m-1}}{[m+1]^2} x_{2r-m-1}^{3r-m-1} + \frac{q^{-m-1}}{[m+1]^2} {a'}_0^{3r-m-3} - \frac{[m+2]}{[m+1]} {b'}_0^{3r-m-3}, \\*
  b_0^{3r-m-2} \otimes a_1^1 &= \frac{\{ m+1 \}'}{[m+1]^3} a_0^{3r-m-1} - \frac{[m]}{[m+1]^2} b_0^{3r-m-1} + \frac{\{ m+1 \}'}{[m+1]^3} \tilde{a}_1^{3r-m-3} \\*
  &\hspace*{\parindent} - \frac{[m+2]}{[m+1]^2} \tilde{b}_1^{3r-m-3},
 \end{align*}  
 and equation~\eqref{E:P_m} gives
 \begin{align*}
  x_0^m &:= x_0^{m-1} \otimes a_0^1, &
  y_0^m &:= \frac{q^{-m}}{[m+1]} a_{2r-m-1}^{m-1} \otimes a_1^1 + \frac{[m]}{[m+1]} y_0^{m-1} \otimes a_0^1,
 \end{align*}
 so equation~\eqref{E:c_d_m} gives
 \begin{align*}
  F_{\tTL}(d_m^+)(b_0^{3r-m-2}) &= \frac{1}{[m-r+1]!} y_0^m, \\*
  F_{\tTL}(d_m^-)(b_0^{3r-m-2}) &= - \frac{1}{[m-r+1]![m+1]} x_0^m. \qedhere
 \end{align*}
\end{proof}

We can now prove Theorem~\ref{T:functor}.

\begin{proof}[Proof of Theorem~\ref{T:functor}]
 For convenience of notation, let us set 
 \[
  T_m :=
  \begin{cases}
   X_m  & 0 \leqs m \leqs r-1 \\
   P_m & r \leqs m \leqs 2r-2
  \end{cases} \in \mods{\bar{U}}
 \]
 for every integer $0 \leqs m \leqs 2r-2$. Let $f$ be a morphism in $\Hom_{\bar{U}}(X^{\otimes m},X^{\otimes m'})$ for some $m,m' \in \tTL$. Thanks to Proposition~\ref{P:domination} we have
 \[
  f = \sum_{n,n'=0}^{2r-2} \sum_{j \in J_{m,n}} \sum_{j' \in J_{m',n'}} F_{\tTL} (v^{m'}_{n',j'} t_{n'} u_{m'}^{n',j'}) \circ f \circ F_{\tTL} (v^m_{n,j} t_n u_m^{n,j}).
 \]
 where $J_{m,n}$ and $J_{m',n'}$ are finite sets, and where 
 \begin{align*}
  u_m^{n,j} &\in \tTL(m,n), & v^m_{n,j} &\in \tTL(n,m), \\
  u_{m'}^{n',j'} &\in \tTL(m',n'), & v^{m'}_{n',j'} &\in \tTL(n',m')
 \end{align*} 
 for all integers $0 \leqs n,n' \leqs 2r-2$ and for all $j \in J_{m,n}$ and $j' \in J_{m',n'}$. Now, thanks to Lemmas~\ref{L:functor} and \ref{L:tilde_functor}, the linear map $F_{\tTL} : \tTL(t_n,t_{n'}) \to \Hom_{\bar{U}}(T_n,T_{n'})$ is surjective for all integers $0 \leqs n,n' \leqs 2r-2$, and therefore
 \[
  F_{\tTL} (t_{n'} u_{m'}^{n',j'}) \circ f \circ F_{\tTL} (v^m_{n,j} t_n) \in \Hom_{\bar{U}}(T_n,T_{n'})
 \]
 belongs to its image. This means that $f$ belongs to the image of the linear map $F_{\tTL} : \tTL(m,m') \to \Hom_{\bar{U}}(X^{\otimes m},X^{\otimes m'})$.
\end{proof}

\section{Additional relations}\label{S:additional_relations}

In this section we list additional relations satisfied by images of generating morphisms $p_{2r-1}^+$, $p_{2r-1}^-$, $i_{2r-1}^+$, $i_{2r-1}^-$ in $\mods{\bar{U}}$. In order to do this, we define the \textit{small Temperley--Lieb category} $\bTL$ as the quotient linear category
\[
 \bTL := \tTL / \ker F_{\tTL}.
\]
In other words, objects of $\bTL$ are the same as objects of $\tTL$, while vector spaces of morphisms of $\bTL$ are quotients of vector spaces of morphisms of $\tTL$ with respect to kernels of linear maps determined by the linear functor $F_{\tTL}$.

\begin{lemma}\label{L:additional_relations}
 In $\bTL$ generating morphisms $p_{2r-1}^\varepsilon$ and $i_{2r-1}^\varepsilon$ satisfy
 \begin{equation}
  \etl{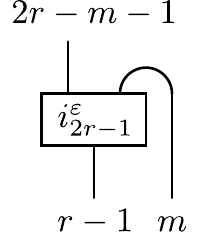} \hspace*{-5pt} = \overline{\varepsilon} [r-1]! \cdot \hspace*{-5pt} \etl{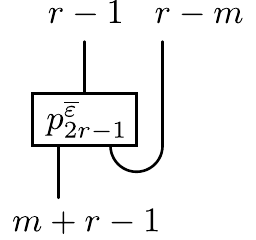} \label{E:cd}
 \end{equation}
 for every integer $0 \leqs m \leqs r$ and every $\varepsilon \in \{ +,- \}$, where $\overline{\varepsilon} \in \{ +,- \}$ denotes the opposite of $\varepsilon$, and they satisfy
 \begin{align}
  \etl{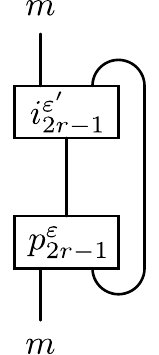} 
  &= (-1)^{m+1} \frac{\delta_{\varepsilon,\varepsilon'}}{[m+1]} \cdot \etl{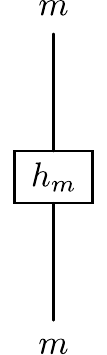} \label{E:pi_ptr} \\* \nonumber \\*
  \etl{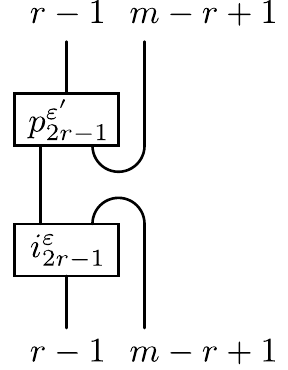} \hspace*{-37.5pt}
  &= (-1)^{m+1} \frac{\delta_{\varepsilon,\varepsilon'}}{[m+1]} \cdot \etl{h_m} \label{E:pi_cut}
 \end{align}
 for every integer $r \leqs m \leqs 2r-2$ and all $\varepsilon,\varepsilon' \in \{ +,- \}$.
\end{lemma}

\begin{proof}
 The result follows directly from equations~\eqref{E:F(c+)} to \eqref{E:F(d-)}.
\end{proof}

\begin{lemma}
 In $\bTL$ generating morphisms $p_{2r-1}^+$, $p_{2r-1}^-$, $i_{2r-1}^+$, $i_{2r-1}^-$ satisfy
 \begin{align}
  \sum_{\varepsilon \in \{ +,- \}} \hspace*{-15pt} \etl{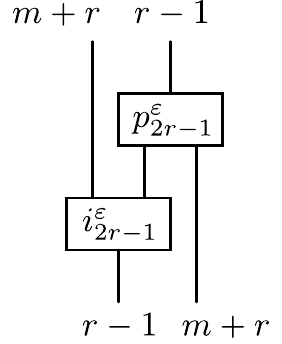} \hspace*{-20pt} &= \etl{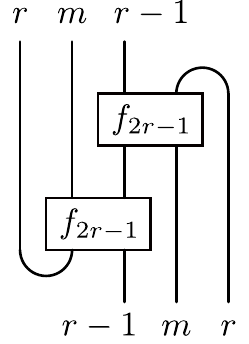} & \hspace*{-5pt} 
  \sum_{\varepsilon \in \{ +,- \}} \hspace*{-20pt} \etl{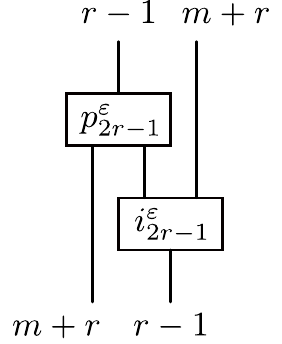} \hspace*{-20pt} &= \etl{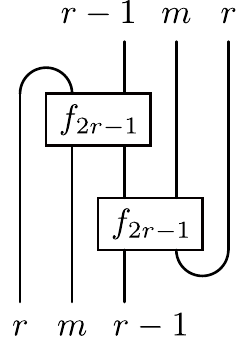} \label{E:i_p_higher}
 \end{align}
 for every integer $0 \leqs m \leqs r$.
\end{lemma}

\begin{proof}
 The result follows from a straightforward computation using equations~\eqref{E:i_p}, \eqref{E:p_i_ambi}, and \eqref{E:cd}.
\end{proof}

\appendix

\section{Computations}\label{A:computations}

In this appendix, we provide computations to be used in Sections~\ref{S:extended_TL_category} and \ref{S:additional_relations}. We start with the identity
\begin{equation}
 [r-j-1]! = (-1)^j \frac{[r-1]!}{[j]!}. \label{E:reversal}
\end{equation}

\begin{lemma}
 For all integers $0 \leqs j \leqs r-1$ we have
 \begin{align}
  \begin{split}
   \Delta(E^j) &= \sum_{k=0}^j \frac{[j]!}{[k]![j-k]!} q^{k(j-k)} E^{j-k} \otimes E^k K^{j-k}, \\
   \Delta(F^j) &= \sum_{k=0}^j \frac{[j]!}{[k]![j-k]!} q^{k(j-k)} F^k K^{-j+k} \otimes F^{j-k}.
  \end{split} \label{E:coproduct}
 \end{align}
\end{lemma}

\begin{proof}
 The statement follows from a direct computation.
\end{proof}

Next, let us fix some notation. For all integers $0 \leqs j \leqs m$, let us set
\[
 V_j^m := \rmspan_\C \{ a_{j_1}^1 \otimes \ldots \otimes a_{j_m}^1 \in X^{\otimes m} \mid j_1 + \ldots + j_m = j \},
\]
and let us give special names to highest and lowest weight vectors
\[
 v_0^m := a_0^1 \otimes \ldots \otimes a_0^1 \in V_0^m, \quad
 v_m^m := a_1^1 \otimes \ldots \otimes a_1^1 \in V_m^m.
\]

\begin{lemma}
 For every integer $0 \leqs m \leqs r-1$ we have
 \begin{align}
  a_0^m &= v_0^m, & a_m^m &= [m]! v_m^m, \label{E:v_0_m_simple}
 \end{align}
 for every integer $r \leqs m \leqs 2r-2$ we have
 \begin{align}
  x_0^m &= v_0^m, & y_{m-r}^m &= \frac{[m-r]![r-1]!}{[m+1]} v_m^m, \label{E:v_0_m_projective}
 \end{align}
 and for $m = 2r-1$ we have
 \begin{align}
  a_0^{r-1,+} &= v_0^{2r-1}, & a_{r-1}^{r-1,-} &= [r-1]!^2 v_{2r-1}^{2r-1}. \label{E:v_0_m_2r-1}
 \end{align}
\end{lemma}

\begin{proof}
 The statement can be proved by induction on $m$.
\end{proof}

\begin{lemma}
 For all integers $0 \leqs m \leqs r-1$ and $0 \leqs j \leqs m$ we have
 \begin{align}
  \begin{split}
   E^j \cdot a_j^m &= \frac{[m]![j]!}{[m-j]!} a_0^m, \\
   F^{m-j} \cdot a_j^m &= a_m^m,
  \end{split} \label{E:highest_lowest_simple}
 \end{align}
 for all integers $r \leqs m \leqs 2r-2$, $0 \leqs k \leqs m-r$, and $0 \leqs \ell \leqs 2r-m-2$ we have
 \begin{align}
  \begin{split}
   E^k \cdot x_k^m &= \frac{[m-r]![k]!}{[m-k-r]!} x_0^m, \\
   E^{m+\ell-r+1} \cdot b_\ell^m &=(-1)^\ell \frac{[m-r]![m+\ell-r+1]![\ell]!}{[m+1]} x_0^m, \\
   F^{m-k-r} \cdot y_k^m &= y_{m-r}^m, \\
   F^{r-\ell-1} \cdot b_\ell^m &= y_{m-r}^m,
  \end{split} \label{E:highest_lowest_projective}
 \end{align}
 and for every integer $0 \leqs j \leqs r-1$ we have
 \begin{align}
  \begin{split}
   E^j \cdot a_j^{r-1,+} &= (-1)^j [j]!^2 a_0^{r-1,+}, \\
   F^{r-j-1} \cdot a_j^{r-1,-} &= a_{r-1}^{r-1,-}.
  \end{split} \label{E:highest_lowest_2r-1}
 \end{align}
\end{lemma}

\begin{proof}
 The statement follows from equations~\eqref{E:reversal} to \eqref{E:v_0_m_2r-1}.
\end{proof}

Next, we recall the definition, for every integer $0 \leqs m \leqs r-1$, of the idempotent endomorphism $\varphi_m \in \End_{\bar{U}}(X^{\otimes m})$ associated with the decomposition specified in equation~\eqref{E:decomposition_simple}, and similarly, for $m = 2r-1$, of the idempotent endomorphisms $\varphi_{2r-1}^+,\varphi_{2r-1}^- \in \End_{\bar{U}}(X^{\otimes 2r-1})$ associated with the decomposition specified in equation~\eqref{E:decomposition_2r-1}. For every integer $0 \leqs j \leqs m \leqs r-1$, we have
\[
 \ker \varphi_m \cap V_j^m = \ker E^j \cap V_j^m, \quad
 \ker \varphi_m \cap V_{m-j}^m = \ker F^j \cap V_{m-j}^m,
\]
and similarly, for every integer $0 \leqs j \leqs m = 2r-1$, we have
\[
 \ker \varphi_{2r-1}^+ \cap V_j^{2r-1} = \ker E^j \cap V_j^{2r-1}, \quad
 \ker \varphi_{2r-1}^- \cap V_{2r-j-1}^{2r-1} = \ker F^j \cap V_{2r-j-1}^{2r-1}.
\]

\begin{lemma}\label{L:computation}
 For all integers $0 \leqs j \leqs m \leqs r-1$ and $0 \leqs k \leqs r-m-1$ we have
 \begin{align}
  \begin{split}
   \varphi_{r-1}(a_j^m \otimes a_k^{r-m-1}) &= (-1)^j \frac{[m+k]!}{[m-j]![j+k]!} q^{-j(m+k+1)} a_{j+k}^{r-1}, \\
   \varphi_{r-1}(a_k^{r-m-1} \otimes a_j^m) &= (-1)^j \frac{[m+k]!}{[m-j]![j+k]!} q^{(m-j)k} a_{j+k}^{r-1},
  \end{split} \label{E:alpha_r-1}
 \end{align}
 for every integer $0 \leqs \ell \leqs m-j-1$ we have
 \begin{align}
  \begin{split}
   \varphi_{2r-1}^+(a_j^m \otimes a_\ell^{2r-m-1}) &= 0, \\
   \varphi_{2r-1}^+(a_j^m \otimes x_k^{2r-m-1}) &= (-1)^j \frac{[m+k]!}{[m-j]![j+k]!} q^{-j(m+k+1)} a_{j+k}^{r-1,+}, \\
   \varphi_{2r-1}^+(a_j^m \otimes y_k^{2r-m-1}) &= 0, \\
   \varphi_{2r-1}^+(a_j^m \otimes b_\ell^{2r-m-1}) &= (-1)^{m+\ell+1} \frac{[m-j-\ell-1]![\ell]!}{[m-j]![m]} q^{-j(\ell+1)} a_{j+\ell-m+r}^{r-1,+}, \\
   \varphi_{2r-1}^+(a_\ell^{2r-m-1} \otimes a_j^m) &= 0, \\
   \varphi_{2r-1}^+(x_k^{2r-m-1} \otimes a_j^m) &= (-1)^j \frac{[m+k]!}{[m-j]![j+k]!} q^{(m-j)k} a_{j+k}^{r-1,+}, \\
   \varphi_{2r-1}^+(y_k^{2r-m-1} \otimes a_j^m) &= 0, \\
   \varphi_{2r-1}^+(b_\ell^{2r-m-1} \otimes a_j^m) &= (-1)^{m+\ell+1} \frac{[m-j-\ell-1]![\ell]!}{[m-j]![m]} q^{-(m-j)(m-\ell)} \mathrlap{a_{j+\ell-m+r}^{r-1,+},}
  \end{split} \label{E:alpha_2r-1_+}
 \end{align}
 and for every integer $m-j \leqs \ell \leqs m-1$ we have
 \begin{align}
  \begin{split}
   \varphi_{2r-1}^-(a_j^m \otimes a_\ell^{2r-m-1}) &= 0, \\
   \varphi_{2r-1}^-(a_j^m \otimes x_k^{2r-m-1}) &= 0, \\
   \varphi_{2r-1}^-(a_j^m \otimes y_k^{2r-m-1}) &= (-1)^{j+1} \frac{[m+k]!}{[m-j]![j+k]![m]} q^{-j(m+k+1)} a_{j+k}^{r-1,-}, \\
   \varphi_{2r-1}^-(a_j^m \otimes b_\ell^{2r-m-1}) &= (-1)^{j+1} \frac{[\ell]!}{[m-j]![j+\ell-m]![m]} q^{-j(\ell+1)} a_{j+\ell-m}^{r-1,-}, \\
   \varphi_{2r-1}^-(a_\ell^{2r-m-1} \otimes a_j^m) &= 0,\\
   \varphi_{2r-1}^-(x_k^{2r-m-1} \otimes a_j^m) &= 0, \\
   \varphi_{2r-1}^-(y_k^{2r-m-1} \otimes a_j^m) &= (-1)^{j+1} \frac{[m+k]!}{[m-j]![j+k]![m]} q^{(m-j)k} a_{j+k}^{r-1,-}, \\
   \varphi_{2r-1}^-(b_\ell^{2r-m-1} \otimes a_j^m) &= (-1)^{j+1} \frac{[\ell]!}{[m-j]![j+\ell-m]![m]} q^{-(m-j)(m-\ell)} \mathrlap{a_{j+\ell-m}^{r-1,+}.}
  \end{split} \label{E:alpha_2r-1_-}
 \end{align}
\end{lemma}

\begin{proof}
 For all integers $0 \leqs m+n \leqs r-1$, $0 \leqs j \leqs m$ and $0 \leqs k \leqs n$, thanks to equation~\eqref{E:highest_lowest_simple}, we have
 \begin{align*}
  \begin{split}
   E^{j+k} \cdot a_{j+k}^{m+n} &= \frac{[m+n]![j+k]!}{[m+n-j-k]!} a_0^{m+n}, \\
   F^{m+n-j-k} \cdot a_{j+k}^{m+n} &= a_{m+n}^{m+n},
  \end{split}
 \end{align*}
 and, thanks to equations~\eqref{E:coproduct} and \eqref{E:v_0_m_simple}, we have
 \begin{align*}
  \begin{split}
   E^{j+k} \cdot a_j^m \otimes a_k^n &= \frac{[m]![n]![j+k]!}{[m-j]![n-k]!} q^{j(n-k)} a_0^{m+n}, \\
   F^{m+n-j-k} \cdot a_j^m \otimes a_k^n &= \frac{[m]![n]![m+n-j-k]!}{[m+n]![m-j]![n-k]!} q^{j(n-k)} a_{m+n}^{m+n}.
  \end{split}
 \end{align*}
 This means
 \[
  \varphi_{m+n}(a_j^m \otimes a_k^n) = \frac{[m]![n]![m+n-j-k]!}{[m+n]![m-j]![n-k]!} q^{j(n-k)} a_{j+k}^{m+n},
 \]
 which implies in particular equation~\eqref{E:alpha_r-1} thanks to equation~\eqref{E:reversal}. For all integers $0 \leqs j \leqs m \leqs r-1$, $0 \leqs k \leqs r-m-1$, and $0 \leqs \ell \leqs m-j-1$, thanks to equation~\eqref{E:highest_lowest_2r-1}, we have
 \begin{align*}
  \begin{split}
   E^{j+k} \cdot a_{j+k}^{r-1,+} &= (-1)^{j+k} [j+k]!^2 a_0^{r-1,+}, \\
   E^{j+\ell-m+r} \cdot a_{j+\ell-m+r}^{r-1,+} &= (-1)^{m+j+\ell+1} \frac{[r-1]!^2}{[m-j-\ell-1]!^2} a_0^{r-1,+},
  \end{split}
 \end{align*}
 and, thanks to equations~\eqref{E:coproduct} to \eqref{E:v_0_m_2r-1}, we have
 \begin{align*}
  \begin{split}
   E^{j+\ell-m+r} \cdot a_j^m \otimes a_\ell^{2r-m-1} &= 0, \\
   E^{j+k} \cdot a_j^m \otimes x_k^{2r-m-1} &= (-1)^k \frac{[m+k]![j+k]!}{[m-j]!} q^{-j(m+k+1)} a_0^{r-1,+}, \\
   E^{j+k+r} \cdot a_j^m \otimes y_k^{2r-m-1} &= 0, \\
   E^{j+\ell-m+r} \cdot a_j^m \otimes b_\ell^{2r-m-1} &= (-1)^j \frac{[\ell]![r-1]!^2}{[m-j-\ell-1]![m-j]![m]} q^{-j(\ell+1)} a_0^{r-1,+}, \\
   E^{j+\ell-m+r} \cdot a_\ell^{2r-m-1} \otimes a_j^m &= 0, \\
   E^{j+k} \cdot x_k^{2r-m-1} \otimes a_j^m &= (-1)^k \frac{[m+k]![j+k]!}{[m-j]!} q^{(m-j)k} a_0^{r-1,+}, \\
   E^{j+k+r} \cdot y_k^{2r-m-1} \otimes a_j^m &= 0, \\
   E^{j+\ell-m+r} \cdot b_\ell^{2r-m-1} \otimes a_j^m &= (-1)^j \frac{[\ell]![r-1]!^2}{[m-j-\ell-1]![m-j]![m]} q^{-(m-j)(m-\ell)} a_0^{r-1,+}.
  \end{split}
 \end{align*}
 For all integers $0 \leqs j \leqs m \leqs r-1$, $0 \leqs k \leqs r-m-1$, and $m-j \leqs \ell \leqs m-1$, thanks to equation~\eqref{E:highest_lowest_2r-1}, we have
 \begin{align*}
  \begin{split}
   F^{r-j-k-1} \cdot a_{j+k}^{r-1,-} &= a_{r-1}^{r-1,-}, \\
   F^{r+m-j-\ell-1} \cdot a_{j+\ell-m}^{r-1,-} &= a_{r-1}^{r-1,-},
  \end{split}
 \end{align*}
 and, thanks to equations~\eqref{E:coproduct} to \eqref{E:v_0_m_2r-1}, we have
 \begin{align*}
  \begin{split}
   F^{r-j-\ell+m-1} \cdot a_j^m \otimes a_\ell^{2r-m-1} &= 0, \\
   F^{2r-j-k-1} \cdot a_j^m \otimes x_k^{2r-m-1} &= 0, \\
   F^{r-j-k-1} \cdot a_j^m \otimes y_k^{2r-m-1} &= (-1)^{j+1} \frac{[m+k]!}{[m-j]![j+k]![m]} q^{-j(m+k+1)} a_{r-1}^{r-1,-}, \\
   F^{r-j-\ell+m-1} \cdot a_j^m \otimes b_\ell^{2r-m-1} &= (-1)^{j+1} \frac{[\ell]!}{[m-j]![j+\ell-m]![m]} q^{-j(\ell+1)} a_{r-1}^{r-1,-}, \\
   F^{r-j-\ell+m-1} \cdot a_\ell^{2r-m-1} \otimes a_j^m &= 0, \\
   F^{2r-j-k-1} \cdot x_k^{2r-m-1} \otimes a_j^m &= 0, \\
   F^{r-j-k-1} \cdot y_k^{2r-m-1} \otimes a_j^m &= (-1)^{j+1} \frac{[m+k]!}{[m-j]![j+k]![m]} q^{(m-j)k} a_{r-1}^{r-1,-}, \\
   F^{r-j-\ell+m-1} \cdot b_\ell^{2r-m-1} \otimes a_j^m &= (-1)^{j+1} \frac{[\ell]!}{[m-j]![j+\ell-m]![m]} q^{-(m-j)(m-\ell)} a_{r-1}^{r-1,-}
  \end{split}
 \end{align*}
\end{proof}

\begin{lemma}
 For all integers $0 \leqs m \leqs r-1$ and $0 \leqs j \leqs m$ we have
 \begin{align}
  \begin{split}
   F^j \cdot a_0^m &= a_j^m, \\
   E^{m-j} \cdot a_m^m &= \frac{[m]![m-j]!}{[j]!} a_j^m, 
  \end{split} \label{E:j_simple}
 \end{align}
 for all integers $r \leqs m \leqs 2r-2$, $0 \leqs k \leqs m-r$, and $0 \leqs \ell \leqs 2r-m-2$ we have
 \begin{align}
  \begin{split}
   F^k \cdot x_0^m &= x_k^m, \\
   F^{m+\ell-r+1} \cdot x_0^m &= a_\ell^m, \\
   E^{r-\ell-1} \cdot y_{m-r}^m &= (-1)^\ell \frac{[m-r]![r-1]!^2}{[m+\ell-r+1]![\ell]![m+1]} a_\ell^m, \\
   E^{m-k-r} \cdot y_{m-r}^m &= \frac{[m-r]![m-k-r]!}{[k]!} y_j^m,
  \end{split} \label{E:k_l_projective}
 \end{align}
 and for every integer $0 \leqs j \leqs r-1$ we have
 \begin{align}
  \begin{split}
   F^j \cdot a_0^{r-1,+} &= a_j^{r-1,+}, \\
   E^{r-j-1} \cdot a_{r-1}^{r-1,-} &= (-1)^j \frac{[r-1]!^2}{[j]!^2} a_j^{r-1,-}.
  \end{split} \label{E:j_2r-1}
 \end{align}
\end{lemma}

\begin{proof}
 The statement follows from a direct computation.
\end{proof}

\begin{lemma}
 For all integers $0 \leqs m \leqs r-1$ and $0 \leqs j \leqs r-1$ we have
 \begin{align}
  \begin{split}
   a_j^{r-1} &= \sum_{k = \max \{ m+j-r+1,0 \}}^{\min \{ m,j \}} \frac{[j]!}{[k]![j-k]!} q^{-(m-k)(j-k)} a_k^m \otimes a_{j-k}^{r-m-1}, \\
   a_j^{r-1} &= \sum_{k = \max \{ m+j-r+1,0 \}}^{\min \{ m,j \}} \frac{[j]!}{[k]![j-k]!} q^{(m+j-k+1)k} a_{j-k}^{r-m-1} \otimes a_k^m, \\
   a_j^{r-1,+} &= \sum_{k=0}^{\max \{ m+j-r,-1 \}} \frac{[j]!}{[k]![j-k]!} q^{-(m-k)(j-k)} a_k^m \otimes a_{j-k+m-r}^{2r-m-1} \\
   &\hspace*{\parindent} + \sum_{k = \max \{ m+j-r+1,0 \}}^{\min \{ m,j \}} \frac{[j]!}{[k]![j-k]!} q^{-(m-k)(j-k)} a_k^m \otimes x_{j-k}^{2r-m-1}, \\
   a_j^{r-1,+} &= \sum_{k=0}^{\max \{ m+j-r,-1 \}} \frac{[j]!}{[k]![j-k]!} q^{(m+j-k+1)k} a_{j-k+m-r}^{2r-m-1} \otimes a_k^m \\
   &\hspace*{\parindent} + \sum_{k = \max \{ m+j-r+1,0 \}}^{\min \{ m,j \}} \frac{[j]!}{[k]![j-k]!} q^{(m+j-k+1)k} x_{j-k}^{2r-m-1} \otimes a_k^m, \\
   a_j^{r-1,-} &= \sum_{k = \max \{ m+j-r+1,0 \}}^{\min \{ m,j \}} - \frac{[j]![m]}{[k]![j-k]!} q^{-(m-k)(j-k)} a_k^m \otimes y_{j-k}^{2r-m-1} \\
   &\hspace*{\parindent} + \sum_{k = \min \{ m+1,j+1 \}}^m \frac{[j]![r-1]!}{[k]![j-k+r]!} q^{-(m-k)(j-k)} a_k^m \otimes a_{j-k+m}^{2r-m-1}, \\
   a_j^{r-1,-} &= \sum_{k = \max \{ m+j-r+1,0 \}}^{\min \{ m,j \}} - \frac{[j]![m]}{[k]![j-k]!} q^{(m+j-k+1)k} y_{j-k}^{2r-m-1} \otimes a_k^m \\
   &\hspace*{\parindent} + \sum_{k = \min \{ m+1,j+1 \}}^m \frac{[j]![r-1]!}{[k]![j-k+r]!} q^{(m+j-k+1)k} a_{j-k+m}^{2r-m-1} \otimes a_k^m.
  \end{split} \label{E:a_j^r-1_+_-}
 \end{align}
\end{lemma}

\begin{proof}
 We have 
 \begin{align*}
  a_j^{r-1} &= F^j \cdot a_0^{r-1} \\*
  &= \sum_{k = \max \{ m+j-r+1,0 \}}^{\min \{ m,j \}} \frac{[j]!}{[k]![j-k]!} q^{k(j-k)} F^k K^{-j+k} \cdot a_0^m \otimes F^{j-k} \cdot a_0^{r-m-1} \\*
  &= \sum_{k = \max \{ m+j-r+1,0 \}}^{\min \{ m,j \}} \frac{[j]!}{[k]![j-k]!} q^{-(m-k)(j-k)} a_k^m \otimes a_{j-k}^{r-m-1},
 \end{align*}
 and
 \begin{align*}
  a_j^{r-1} &= F^j \cdot a_0^{r-1} \\*
  &= \sum_{k = \max \{ m+j-r+1,0 \}}^{\min \{ m,j \}} \frac{[j]!}{[k]![j-k]!} q^{k(j-k)} F^{j-k} K^{-k} \cdot a_0^{r-m-1} \otimes F^k \cdot a_0^m \\*
  &= \sum_{k = \max \{ m+j-r+1,0 \}}^{\min \{ m,j \}} \frac{[j]!}{[k]![j-k]!} q^{(m+j-k+1)k} a_{j-k}^{r-m-1} \otimes a_k^m.
 \end{align*}
 Similarly, we have 
 \begin{align*}
  a_j^{r-1} &= (-1)^j \frac{[j]!^2}{[r-1]!^2} E^{r-j-1} \cdot a_{r-1}^{r-1} \\*
  &= \sum_{k = \max \{ m+j-r+1,0 \}}^{\min \{ m,j \}} (-1)^{j+k} \frac{[j]![m+j-k]!}{[m-k]![r-1]!^2} q^{-(m-k)(m+j-k+1)} \\*
  &\hspace*{\parindent} E^{m-k} \cdot a_m^m \otimes E^{r-m-j+k-1} K^{m-k} \cdot a_{r-m-1}^{r-m-1} \\*
  &= \sum_{k = \max \{ m+j-r+1,0 \}}^{\min \{ m,j \}} \frac{[j]!}{[k]![j-k]!} q^{-(m-k)(j-k)} a_k^m \otimes a_{j-k}^{r-m-1}
 \end{align*}
 and 
 \begin{align*}
  a_j^{r-1} &= (-1)^j \frac{[j]!^2}{[r-1]!^2} E^{r-j-1} \cdot a_{r-1}^{r-1} \\*
  &= \sum_{k = \max \{ m+j-r+1,0 \}}^{\min \{ m,j \}} (-1)^{j+k} \frac{[j]![m+j-k]!}{[m-k]![r-1]!^2} q^{-(m-k)(m+j-k+1)} \\*
  &\hspace*{\parindent} E^{r-m-j+k-1} \cdot a_{r-m-1}^{r-m-1} \otimes E^{m-k} K^{-m-j+k-1} \cdot a_m^m \\*
  &= \sum_{k = \max \{ m+j-r+1,0 \}}^{\min \{ m,j \}} \frac{[j]!}{[k]![j-k]!} q^{(m+j-k+1)k} a_{j-k}^{r-m-1} \otimes a_k^m.
 \end{align*}
 Similarly, we have 
 \begin{align*}
  a_j^{r-1,+} &= F^j \cdot a_0^{r-1,+} \\*
  &= \sum_{k=0}^{\min \{ m,j \}} \frac{[j]!}{[k]![j-k]!} q^{k(j-k)} F^k K^{-j+k} \cdot a_0^m \otimes F^{j-k} \cdot x_0^{2r-m-1} \\*
  &= \sum_{k=0}^{\max \{ m+j-r,-1 \}} \frac{[j]!}{[k]![j-k]!} q^{-(m-k)(j-k)} a_k^m \otimes a_{j-k+m-r}^{2r-m-1} \\*
  &\hspace*{\parindent} + \sum_{k = \max \{ m+j-r+1,0 \}}^{\min \{ m,j \}} \frac{[j]!}{[k]![j-k]!} q^{-(m-k)(j-k)} a_k^m \otimes x_{j-k}^{2r-m-1}
 \end{align*}
 and
 \begin{align*}
  a_j^{r-1,+} &= F^j \cdot a_0^{r-1,+} \\*
  &= \sum_{k=0}^{\min \{ m,j \}} \frac{[j]!}{[k]![j-k]!} q^{k(j-k)} F^{j-k} K^{-k} \cdot x_0^{2r-m-1} \otimes F^k \cdot a_0^m \\*
  &= \sum_{k=0}^{\max \{ m+j-r,-1 \}} \frac{[j]!}{[k]![j-k]!} q^{(m+j-k+1)k} a_{j-k+m-r}^{2r-m-1} \otimes a_k^m \\*
  &\hspace*{\parindent} + \sum_{k = \max \{ m+j-r+1,0 \}}^{\min \{ m,j \}} \frac{[j]!}{[k]![j-k]!} q^{(m+j-k+1)k} x_{j-k}^{2r-m-1} \otimes a_k^m.
 \end{align*}
 Similarly, we have 
 \begin{align*}
  a_j^{r-1,-} &= (-1)^j \frac{[j]!^2}{[r-1]!^2} E^{r-j-1} \cdot a_{r-1}^{r-1,-} \\*
  &= \sum_{k = \max \{ m+j-r+1,0 \}}^m (-1)^{j+k+1} \frac{[j]![m+j-k]![m]}{[m-k]![r-1]!^2} q^{-(m-k)(m+j-k+1)} \\*
  &\hspace*{\parindent} E^{m-k} \cdot a_m^m \otimes E^{r-m-j+k-1} K^{m-k} \cdot y_{r-m-1}^{2r-m-1} \\*
  &= \sum_{k = \max \{ m+j-r+1,0 \}}^{\min \{ m,j \}} - \frac{[j]![m]}{[k]![j-k]!} q^{-(m-k)(j-k)} a_k^m \otimes y_{j-k}^{2r-m-1} \\*
  &\hspace*{\parindent} + \sum_{k = \min \{ m+1,j+1 \}}^m \frac{[j]![r-1]!}{[k]![j-k+r]!} q^{-(m-k)(j-k)} a_k^m \otimes a_{j-k+m}^{2r-m-1}
 \end{align*} 
 and
 \begin{align*}
  a_j^{r-1,-} &= (-1)^j \frac{[j]!^2}{[r-1]!^2} E^{r-j-1} \cdot a_{r-1}^{r-1,-} \\*
  &= \sum_{k = \max \{ m+j-r+1,0 \}}^m (-1)^{j+k+1} \frac{[j]![m+j-k]![m]}{[m-k]![r-1]!^2} q^{-(m-k)(m+j-k+1)} \\*
  &\hspace*{\parindent} E^{r-m-j+k-1} \cdot y_{r-m-1}^{2r-m-1} \otimes E^{m-k} K^{-m-j+k-1} \cdot a_m^m \\*
  &= \sum_{k = \max \{ m+j-r+1,0 \}}^{\min \{ m,j \}} - \frac{[j]![m]}{[k]![j-k]!} q^{(m+j-k+1)k} y_{j-k}^{2r-m-1} \otimes a_k^m \\*
  &\hspace*{\parindent} + \sum_{k = \min \{ m+1,j+1 \}}^m \frac{[j]![r-1]!}{[k]![j-k+r]!} q^{(m+j-k+1)k} a_{j-k+m}^{2r-m-1} \otimes a_k^m. \qedhere
 \end{align*}
\end{proof}

\end{document}